\documentclass[]{amsart} \usepackage{amsfonts}
\usepackage{bbding} \usepackage[dvips]{graphicx} \usepackage{amssymb,amsmath}
\usepackage{amsbsy,amscd}
\usepackage{xcolor}
\usepackage{enumerate}
\usepackage{lmodern}
\usepackage{stmaryrd}
\usepackage[all, knot]{xy}
\usepackage{leftidx}
\usepackage[colorlinks=true,
pdfstartview=FitE, linkcolor=green!70!black, citecolor=brown,
urlcolor=blue]{hyperref}

\newcommand{\A}{\mathbf{A}}
\newcommand{\Fix}{\mathrm{Fix}}
\newcommand{\cyp}[2]{{#1}({#2})}
\DeclareMathOperator{\Isom}{Isom} \DeclareMathOperator{\PGL}{PGL}
 
\DeclareMathOperator{\diag}{diag} 
\DeclareMathOperator{\Hom}{Hom} \DeclareMathOperator{\diam}{diam}
\DeclareMathOperator{\St}{St} 
\newtheorem{theorem}{Theorem}
\newtheorem{proposition}{Proposition}
\newtheorem{lemma}{Lemma}
\newtheorem{corollary}{Corollary}
\newtheorem{definition}{Definition}
\newtheorem*{theorem*}{Theorem}
\theoremstyle{remark}

\title{On the Hilbert geometry of simplicial Tits sets}
\date{}
\author[X. Nie]{Xin Nie} \address{Universit\'e Paris-Sud, Laboratoire de Math\'ematiques, Orsay F-91405 Cedex} \email{xin.nie@math.u-psud.fr}

\thanks{The research leading to these results has received funding from the European Research Council under the {\em European Community}'s seventh Framework Programme (FP7/2007-2013)/ERC {\em  grant agreement}}

\begin{document}

\begin{abstract}
The moduli space of convex projective structures on a simplicial hyperbolic
 Coxeter orbifold is either a point or the real line. Answering a question of M. Crampon, we prove that in the
latter case, when one goes to infinity in the moduli space,
the entropy of the Hilbert metric tends to $0$.
\end{abstract}

\maketitle

\tableofcontents
\section{Statement of the results}

Let $\mathbb{P}^n$ be the real projective space of dimension $n$. A real projective structure on a manifold, or more generally an orbifold, is an atlas which patches open sets of $\mathbb{P}^n$ together by projective transformations. 

An extensively studied class of projective structures (see \cite{benoist_divisible} and the references therein) comes from the following construction. An open subset $\Omega\subset\mathbb{P}^n$ is said to by \emph{properly convex} if it is a bounded convex subset of an affine chart $\mathbb{R}^n\subset\mathbb{P}^n$. Let $X=\widetilde{X}/\Pi$ be an orbifold, where $\widetilde{X}$ is homeomorphic to $\mathbb{R}^n$ and $\Pi$ is a
group acting discontinuously on $\widetilde{X}$. A \emph{properly convex
projective structure} on $X$ consists of a faithful representation 
$\rho:\Pi\rightarrow\PGL_{n+1}\mathbb{R}$ and a properly convex open set $\Omega\subset\mathbb{P}^n$, such that there is a $\rho$-equivariant homeomorphism $\widetilde{X}\rightarrow\Omega$. 


$\Omega$ is determined by $\rho$ in the sense that $\Omega$ is the convex hull of any $\rho(\Pi)$-orbit in $\Omega$ \cite{vey}, so the \emph{moduli space of properly convex projective structures}
$\mathfrak{P}(X)$  is defined as the subset in the moduli space of representations $$\mathfrak{P}(X)\subset\Hom(\Pi,\PGL_{n+1}\mathbb{R})/\PGL_{n+1}\mathbb{R}$$ given by those $\rho\in\Hom(\Pi,\PGL_{n+1}\mathbb{R})$ which arise from properly convex projective structures.  It is known that $\mathfrak{P}(X)$ is an open and closed subset in the moduli space of representations  \cite{benoist_3} and is homeomorphic to
$\mathbb{R}^{16g-16}$ when $X$ is a closed oriented surface of genus $g$ \cite{goldman}.

In this article we study the case where $X$ is a hyperbolic
simplicial Coxeter orbifold, namely, $X=\mathbb{H}^n/\Gamma$, where
$\mathbb{H}^n$ is the real hyperbolic $n$-space and
$\Gamma\subset\Isom(\mathbb{H}^n)$ is generated by orthogonal
reflections with respect to the faces of a bounded $n$-simplex $P$ in $\mathbb{H}^n$ such that $P$
is a fundamental domain of $\Gamma$. Up to conjugacy, $\Gamma$ is uniquely determined by a hyperbolic Coxeter diagram $J$, so we denote $X$ by $X_J$. Hyperbolic Coxeter diagrams are classified by F. Lann\'er \cite{flanner} as in Figure \ref{lanner}. In particular, they exist only when $n\leq 4$. We divide them into two classes: circular and non-circular ones.
\begin{figure}[ht]
\includegraphics[width=4in]{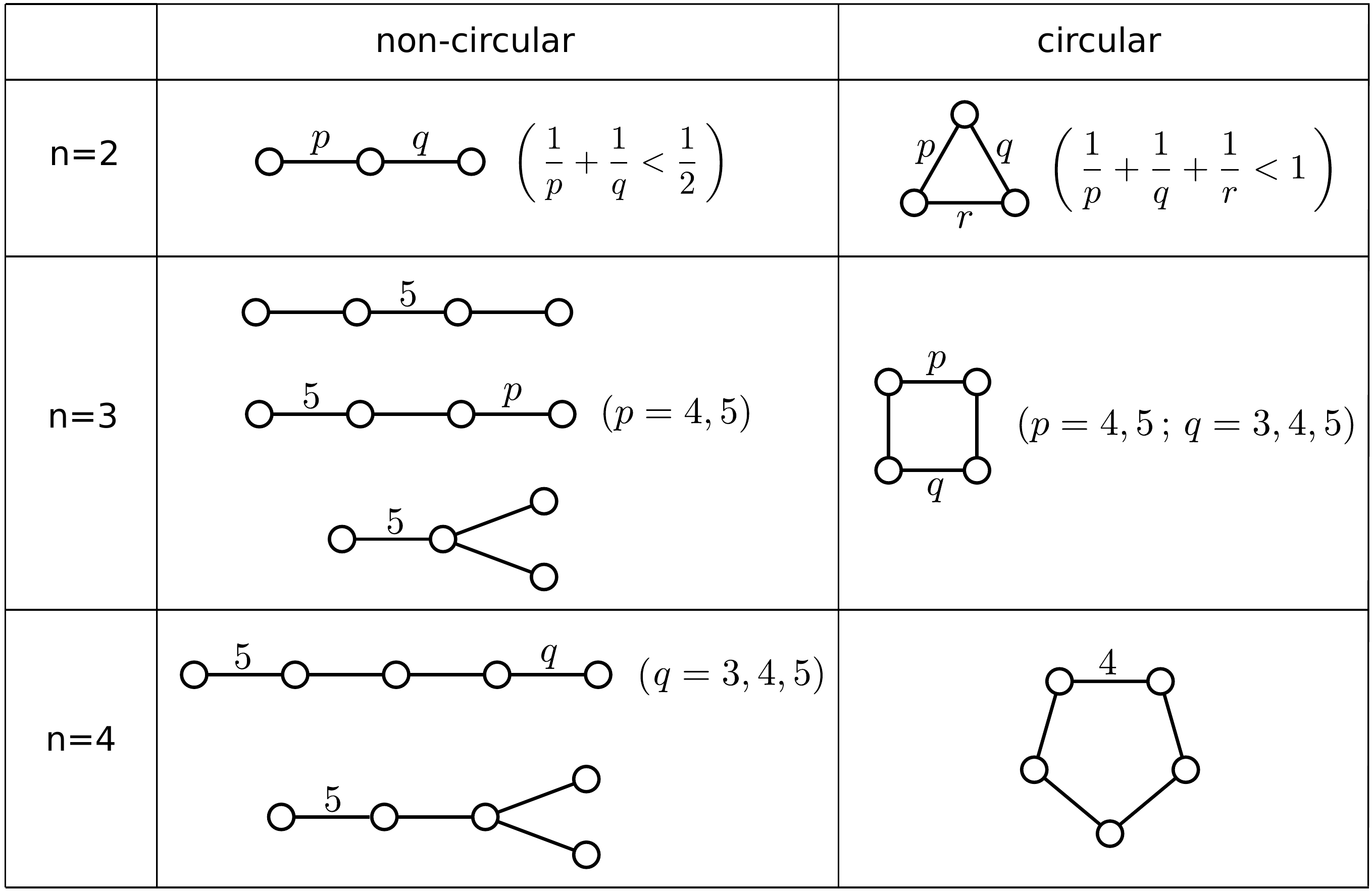}
\caption{All hyperbolic Coxeter diagrams as classified by F. Lann\'er. Here each edge without specified weight  has weight $3$.
} \label{lanner}
\end{figure}

The following result should be well-known to specialists and is stated in \cite{goldman0} in the two-dimensional case.

\begin{proposition}\label{goldman}
Let $J$ be a hyperbolic Coxeter diagram, then
$$
\mathfrak{P}(X_J)\cong\left\{
\begin{array}{cl}
\mathbb{R}_+&\mbox{ if } J\mbox{ is circular,}\\
 \mbox{a point } &\mbox{ otherwise.}
\end{array}
\right.
$$
\end{proposition}

A proof of Proposition \ref{goldman} is given in Section \ref{section_moduli}.

For a circular hyperbolic Coxeter diagram $J$, we shall study how the convex set $\Omega$ deforms as $\rho$
goes to $0$ or $+\infty$ in $\mathfrak{P}(X_J)\cong\mathbb{R}_+$.
\begin{proposition}\label{shape}
Let $P$ be a simplex in $\mathbb{P}^n$. Let
$X_J=\mathbb{H}^n/\Gamma$ be a hyperbolic simplicial Coxeter orbifold given by a circular diagram $J$. Then there exists a
one-parameter family of representations
$\{\rho_t\}_{t\in\mathbb{R}_+}$ of $\Gamma$ into
$\PGL_{n+1}\mathbb{R}$, such that

(1) Each $\rho_t(\Gamma)$ is generated by projective reflections with
respect to faces of $P$.

(2) The map $\mathbb{R}_+\rightarrow \mathfrak{P}(X_J)$, $t\mapsto [\rho_t]$ is bijective.

(3) Let $\Omega_t=\bigcup_{\gamma\in\Gamma}\rho_t(P)\subset\mathbb{P}^n$ be the properly convex open set preserved by $\rho_t$. Then $\Omega_t$ converges to $P$ in Hausdorff topology when $t$
tends to $0$ or $+\infty$. 
\end{proposition}
See Figure \ref{limit} for a $2$-dimensional example.

\begin{figure}
\centering
\includegraphics[width=1.7	in]{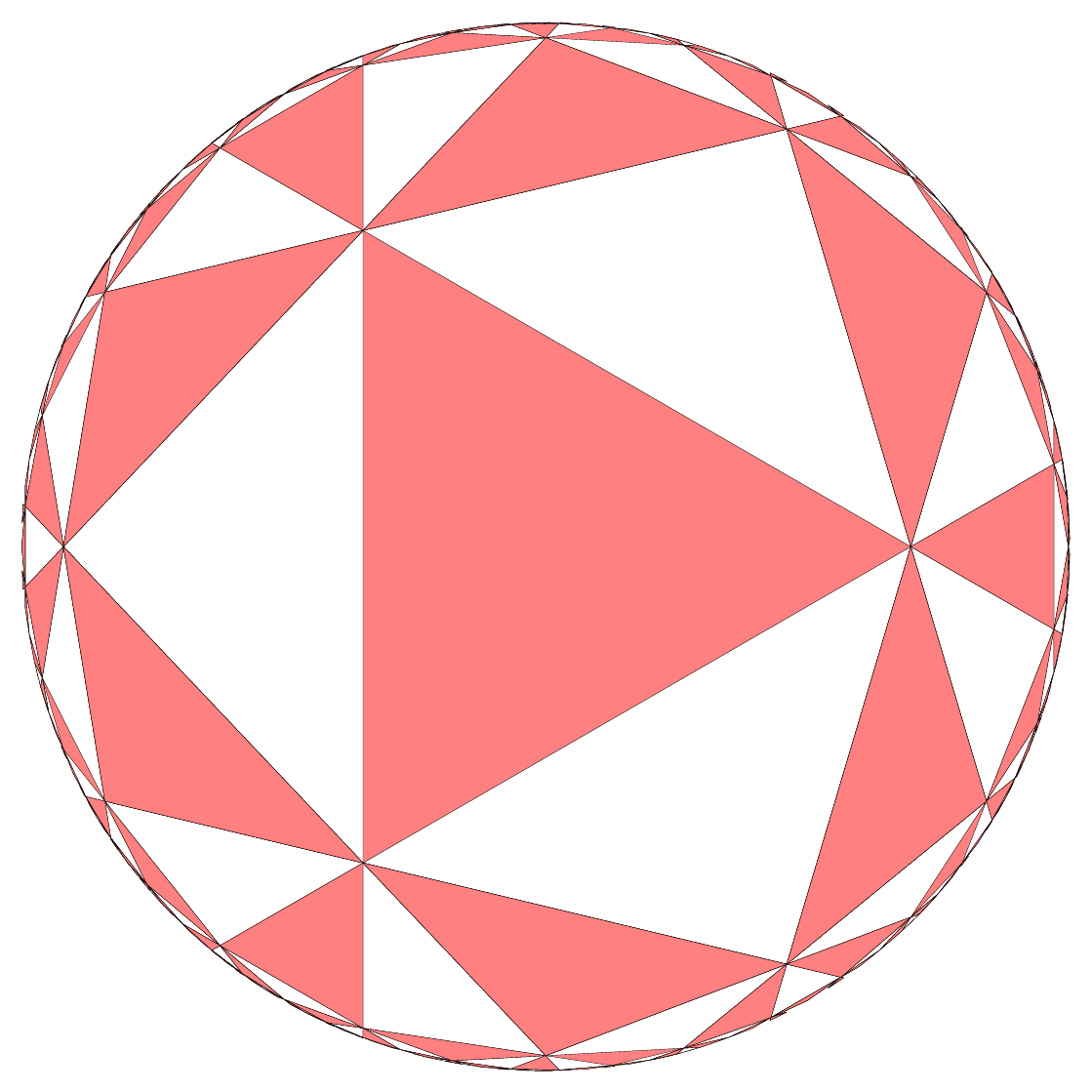}\\
\includegraphics[width=1.65in]{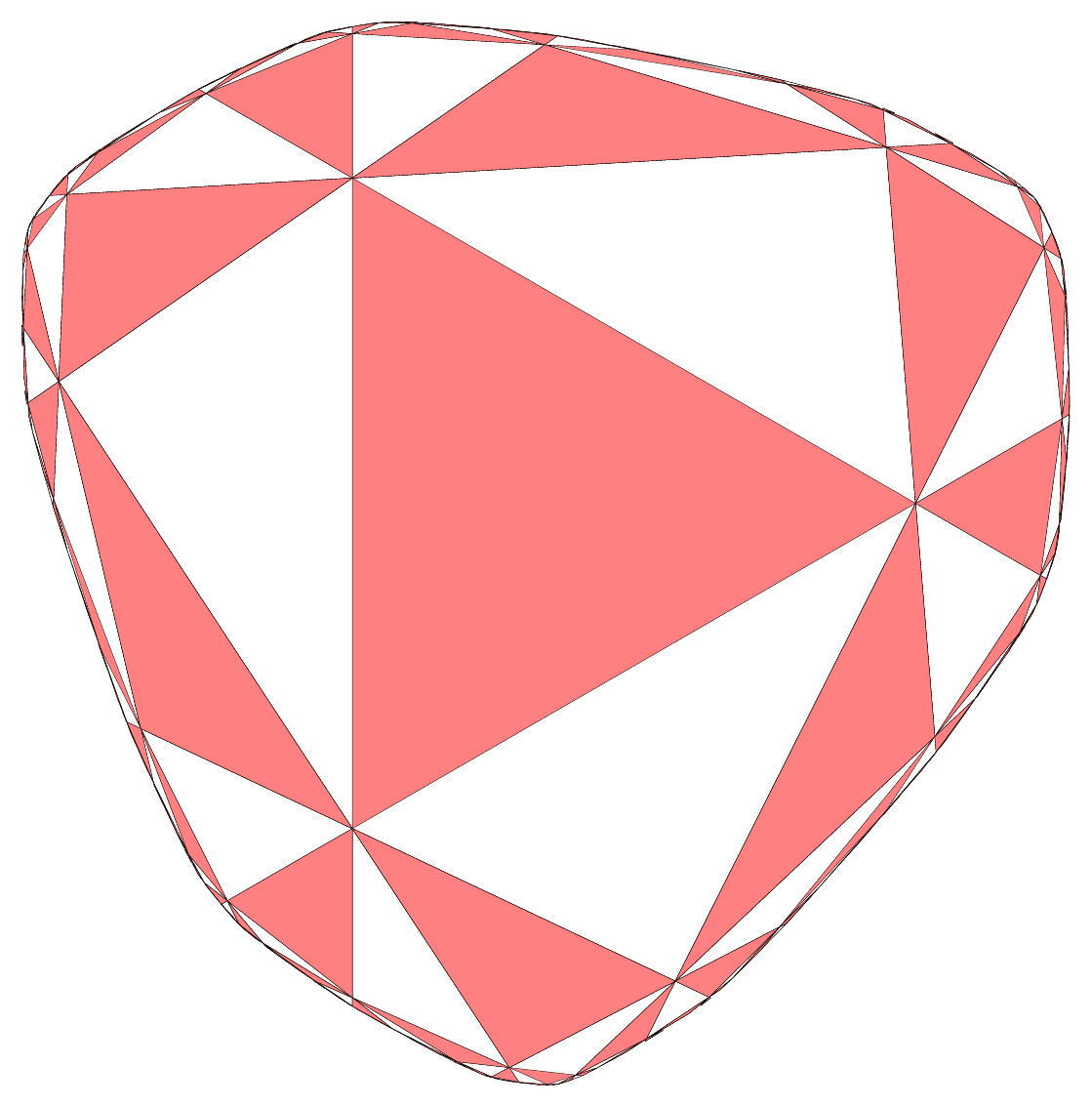}\hspace{0.5in}
\includegraphics[width=1.65in]{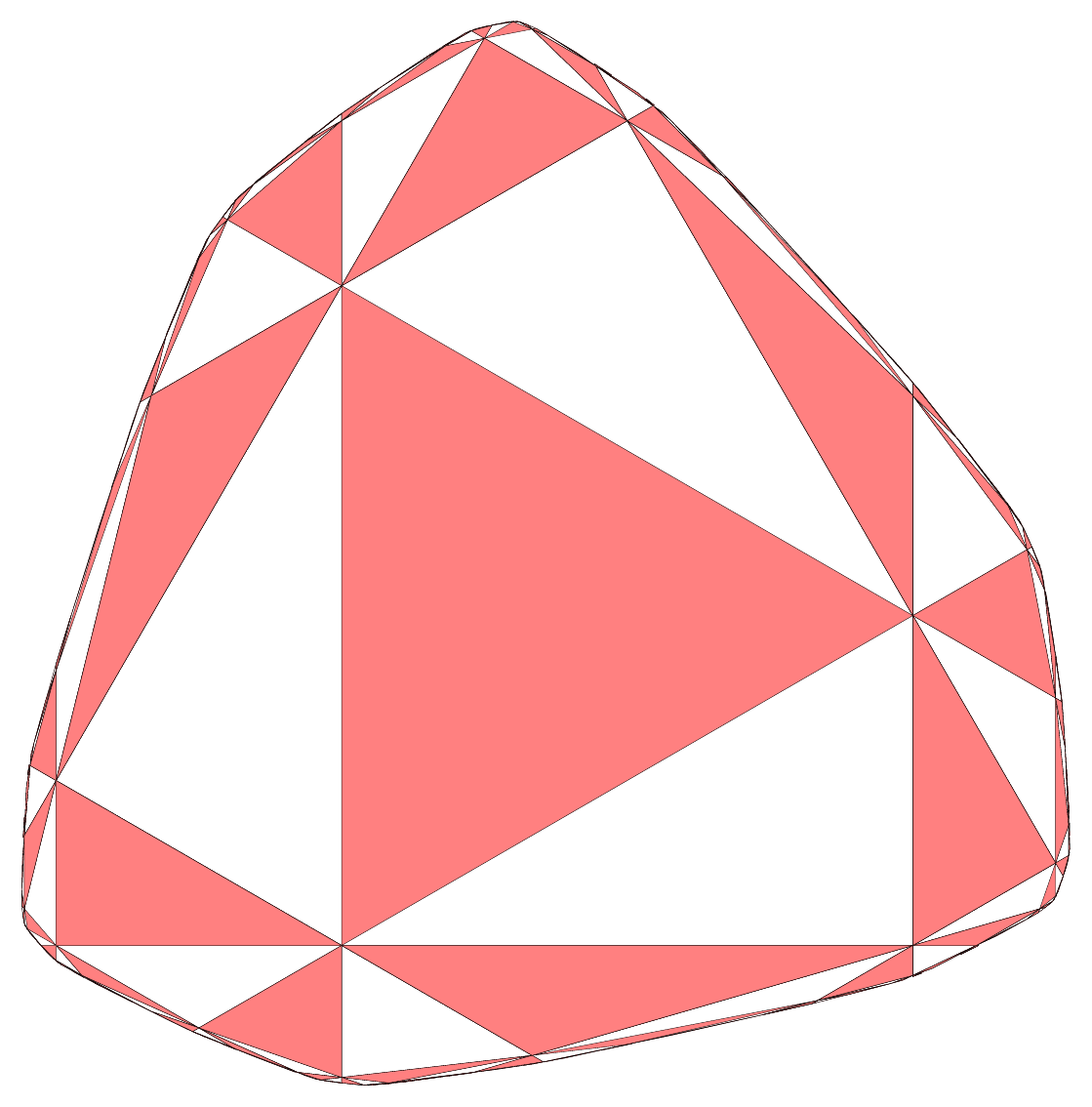}\\
\hspace{0.4in}
\includegraphics[width=1.55in]{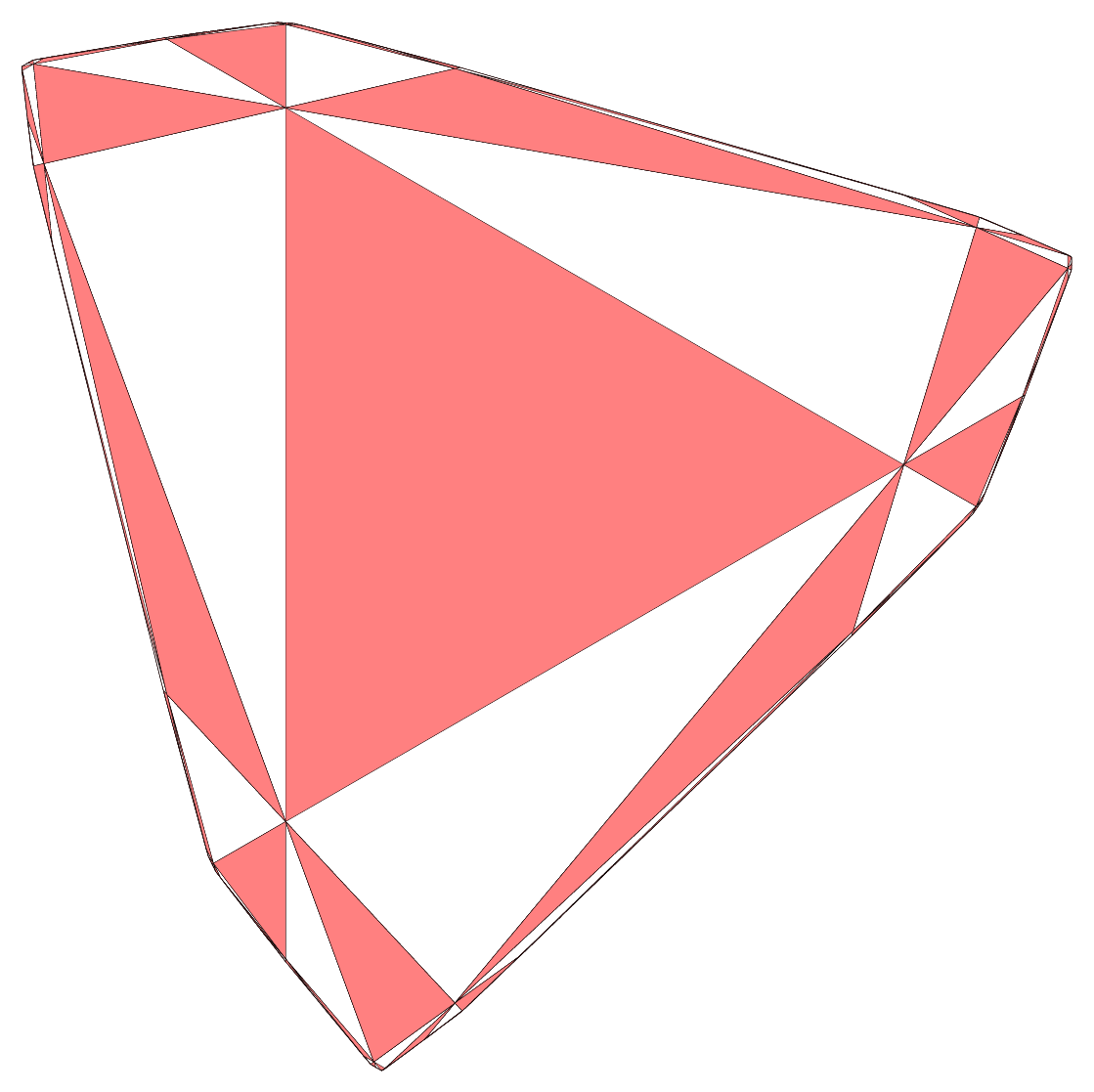}\hspace{0.85in}
\includegraphics[width=1.55in]{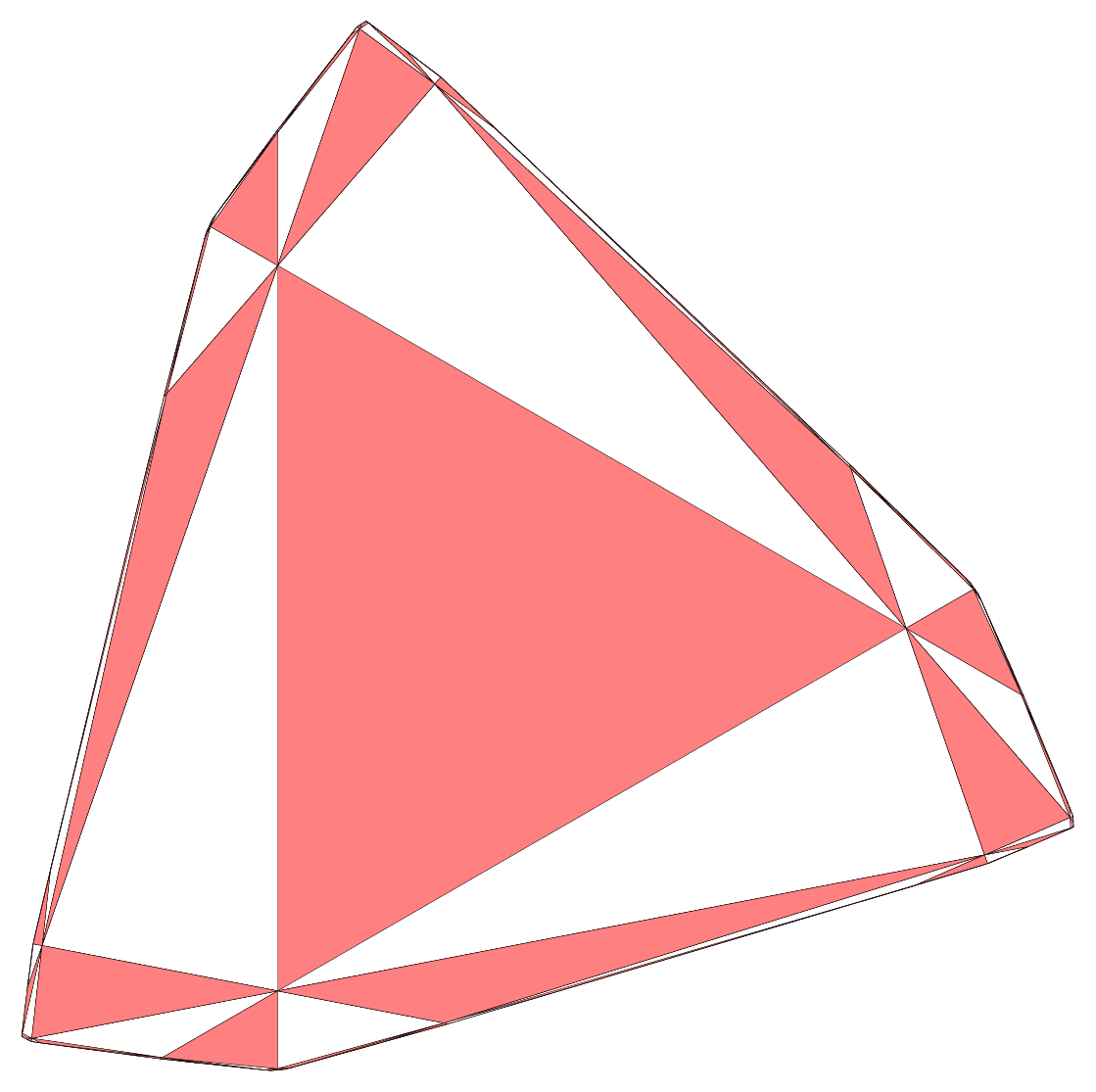}\\
\hspace{0.4in}
\includegraphics[width=1.1in]{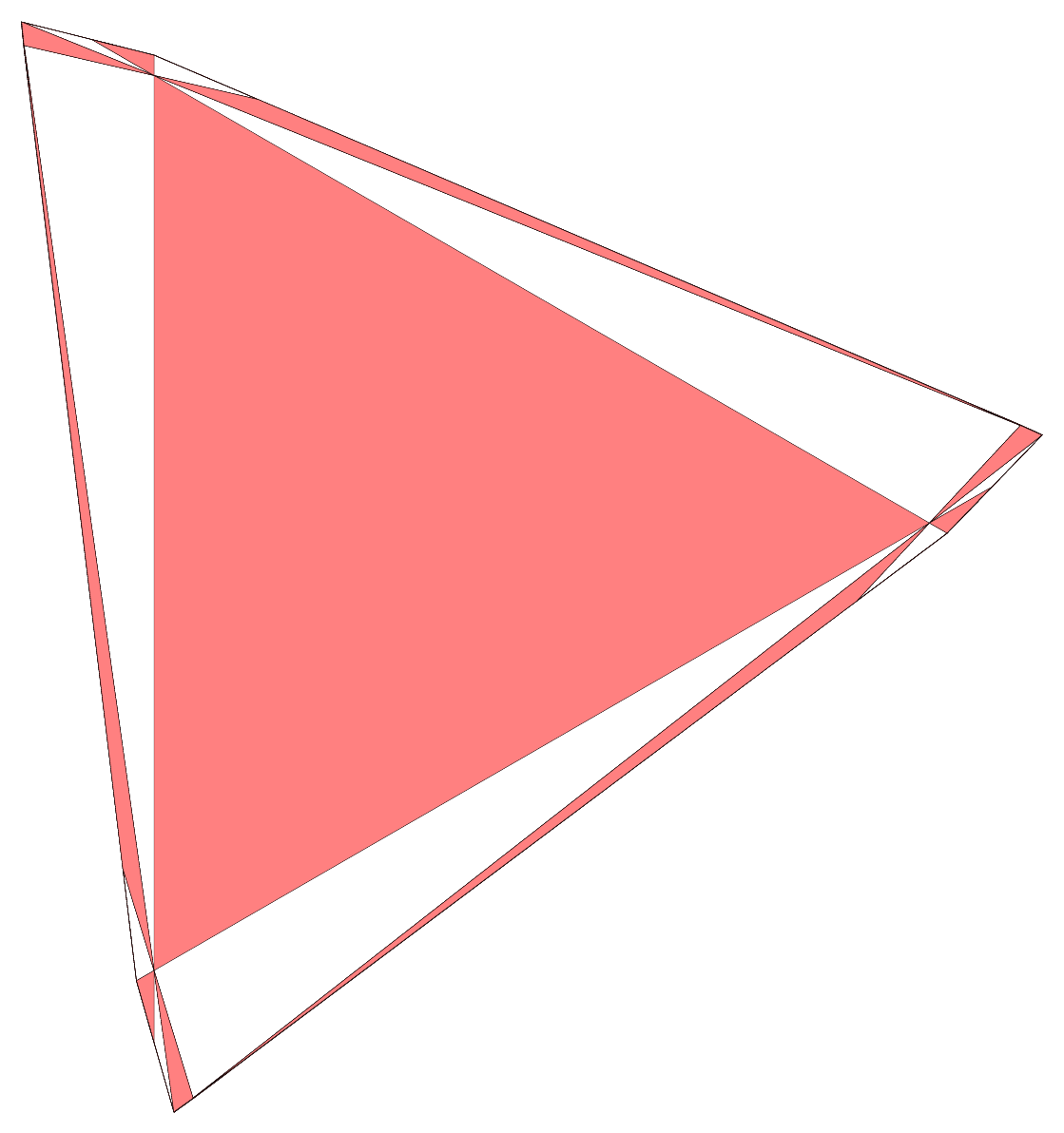}\hspace{1.6in}
\includegraphics[width=1.1in]{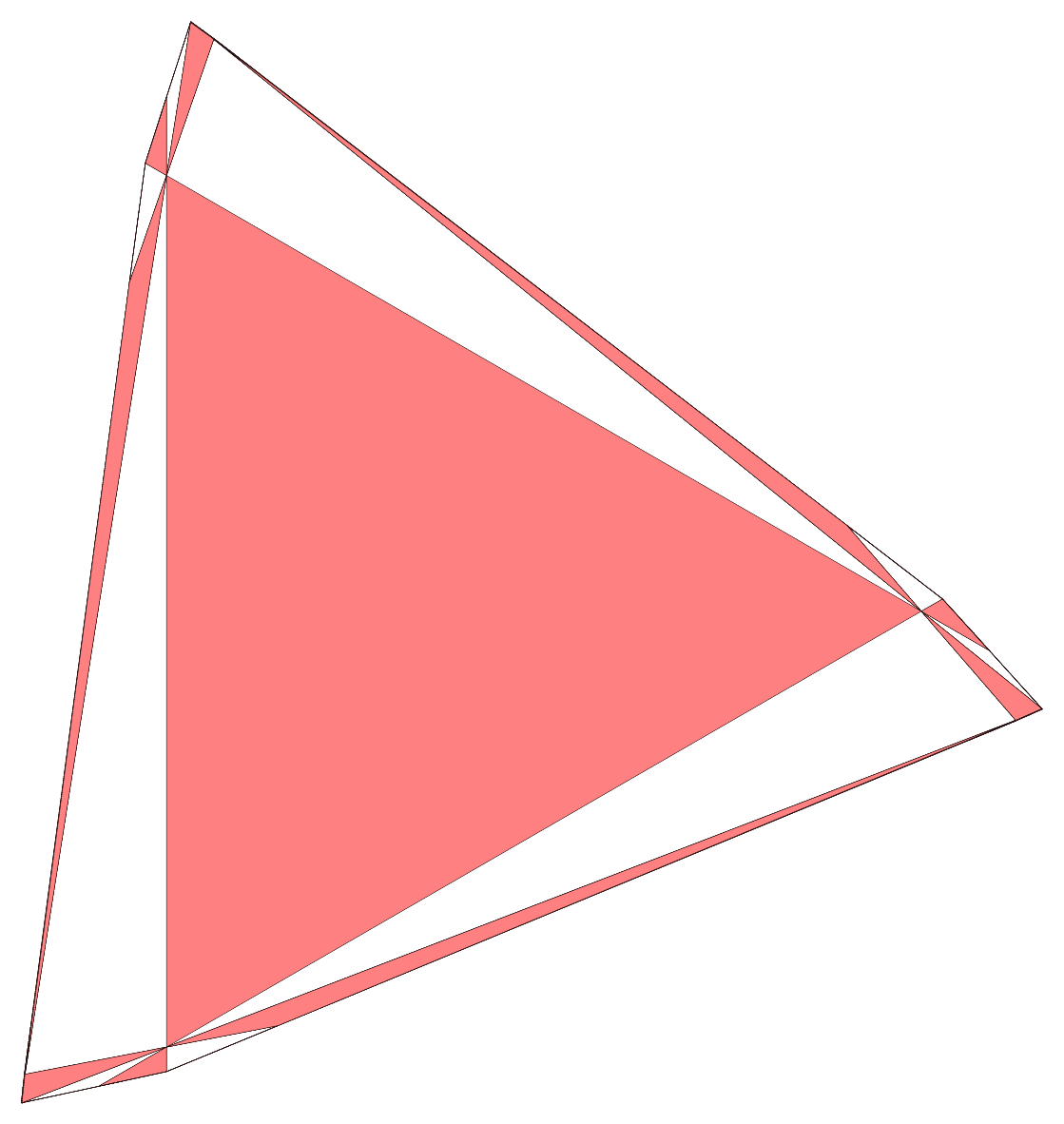}
\caption{Deformation of $\Omega_t$ when $t$ tends to $0$ and
$+\infty$.}\label{limit}
\end{figure}

Our main result concerns the metric geometry on the above family of
convex sets. Any properly convex open set
$\Omega\subset\mathbb{P}^n$ carries a canonical Finsler metric
$d_\Omega$, called the \emph{Hilbert metric}, which is invariant under projective transformations preserving $\Omega$. If $\Omega$ is an
 ellipsoid, then
$(\Omega,d_\Omega)$ is isometric to the real hyperbolic $n$-space
$\mathbb{H}^n$. 

Consider the convex set $\Omega_t$ from Proposition \ref{shape}. We denote its Hilbert metric by $d_t$. By projective invariance, $d_t$ induces a metric on $X_J$. From Proposition \ref{shape} we can
readily deduce some easy geometric properties of the family of metrics $\{d_t\}$. For
example, the diameter of $(X_J, d_t)$
tends to infinity as $t\rightarrow 0$ or $t\rightarrow+\infty$. The purpose of this paper is to study a more subtle quantity, the entropy.

\begin{definition}
Let $(\widetilde{X},d)$ be a metric space and $\Gamma$ be a group acting properly discontinuously on $\widetilde{X}$ by isometries. Given a base point $x_0\in\widetilde{X}$, the (exponential) growth rate of the orbit $\Gamma.x_0$ is defined as
$$
\delta(\widetilde{X}, d, \Gamma, x_0)=\varlimsup_{R\rightarrow +\infty}\frac{1}{R}  \log\#\big(\Gamma.x_0\cap
B(x_0, R)\big),
$$
where  $B(x_0,R)$ is the ball of radius $R$ centered at $x_0$.
\end{definition}
This notion originally arose from the case where $(\widetilde{X},d)$ is the universal covering of a compact non-positively curved Riemaniann manifold $X$ and $\Gamma=\pi_1(X)$. In this case the above growth rate is independent of the choice of $x_0$ and equals the topological entropy of the geodesic flow on the unit tangent bundle of $X$ \cite{manning}. This result easily generalizes to geodesic flows of compact convex projective manifolds endowed with the Hilbert metric (see \cite{crampon}). For this reason, we refer to the orbit growth rate as the \emph{entropy} and omit $x_0$ in the notation.

For any properly convex open set $\Omega\subset\mathbb{P}^n$ acted upon by a discrete group $\Gamma\subset\PGL_{n+1}\mathbb{R}$ with compact fundamental domain, M. Crampon \cite{crampon} proved that the entropy is bounded from above
$$\delta(\Omega,d_\Omega,\Gamma)\leq n-1$$
and the equality is achieved if and only if $\Omega$ is an ellipsoid. He then asked whether $\delta(\Omega,d_\Omega,\Gamma)$ has a lower bound.

Our main result gives a negative answer:
\begin{theorem}\label{thm1}
Let $X_J=\mathbb{H}^n/\Gamma$ be a hyperbolic simplicial Coxeter orbifold with
$J$ circular. Let
$\rho_t$ and $\Omega_t$ be given by Proposition
\ref{shape} and $d_t$ be the Hilbert metric on $\Omega_t$. Then
$$
\delta(\Omega_t, d_t,\Gamma)\rightarrow 0 \mbox{ as } t\rightarrow 0\mbox{ or }+\infty.
$$
\end{theorem}

The main ingredient in the proof of Theorem \ref{thm1} is the
following result.
\begin{lemma}\label{main}
There exists a constant $C$ depending only on $J$,
such that if $A$ and $B$ are two $k$-dimensional cells of the simplex $P$ and
$E=A\cap B$ is a $(k-1)$-dimensional cell, where $1\leq k\leq n-1$, then we have
$$
Cd_t(x,y)\geq d_t(x, E)+d_t(y, E)
$$
for any $x\in A$, $y\in B$ and $t\in\mathbb{R}_+$.
\end{lemma}

As another consequence of Lemma \ref{main}, we construct
families of convex projective structures on surfaces which answer Crampon's problem and have some other curious properties.

\begin{corollary}\label{surface}
On an oriented closed surface of genus $g\geq 2$, there
exists an one-parameter family of convex projective structures such that when the parameter goes to infinity, the entropy of Hilbert metric tends to
$0$, whereas the systole and constant of Gromov hyperbolicity
tends to $+\infty$.
\end{corollary}

Recall that for a metrized manifold $(X,d)$, the \emph{systole} is defined as the infimum
of lengths of homotopically non-trivial closed curves on $X$. Let $\widetilde{X}$ be the universal covering of $X$, then the
\emph{constant of Gromov hyperbolicity} is defined to be the
supremum of sizes of geodesic triangles in $\widetilde{X}$. Here the
``size" of a geodesic triangle $\Delta$ is the minimal
perimeter of all geodesic triangles inscribed in $\Delta$ (\textit{c.f.} \cite{ghys}).

The paper is organized as follows. After recalling 
some backgrounds about reflection groups in Section \ref{section_pre}, we prove Proposition \ref{goldman}
and \ref{shape} in Section \ref{section_moduli}. In Section \ref{section_metric} we prove Theorem \ref{thm1} and Corollary \ref{surface} assuming
Lemma \ref{main}. Finally we prove Lemma \ref{main} in
Section \ref{section_proof}.

\subsection*{Acknowledgements} 
This paper was written when the author was a Ph.D student at Universit\'e Pierre et Marie Curie. The author is grateful to Gilles Courtois for his guidance and to Micka\"el Crampon for asking the question which we deal here.

\section{Preliminaries}\label{section_pre}

In this section we recall some well-known facts about reflection
groups and Tits set. See \cite{benoist_five, vinberg} for details.

A projective transformation $s\in\PGL_{n+1}\mathbb{R}$ is called a  \emph{reflection} if it is conjugate to $\pm\diag
(-1,1,\cdots,1)$. The fixed point set of $s$ is $\Fix(s)=F\sqcup f$ for some hyperplan $F\subset\mathbb{P}^n$ and 
some point $f\notin F$. Reflections
are in one-to-one correspondence with pairs $(f,F)$ with $f\notin F$.

Let $P$ be a $n$-dimesnional simplex in $\mathbb{P}^n$ with faces
$P_i$, where $i=0,\cdots,n$. We choose a
reflection $s_i$ with respect to each $P_i$. We are interested in
the group $\Gamma\subset\PGL_{n+1}\mathbb{R}$ generated by the
$s_i$'s and call it a \emph{simplicial reflection
group}, and call $P$ the \emph{fundamental simplex}. 
Since simplices in $\mathbb{P}^n$ are
conjugate to each other by projective transformations, when studying $\Gamma$ up to conjugacy, 
we can assume 
$$
P=\{[x_0:\cdots:x_n]\in\mathbb{P}^n\mid x_i\geq 0, \forall i\}
$$ 
with faces $P_i=\{[x_0:\cdots:x_n]\mid x_i=0,x_k\geq 0,
\forall k\neq i \}$, so that $\Gamma$ is determined by  $f_1,\cdots, f_n\in\mathbb{P}^n$, where $f_i\notin P_i$ is a fixed point of $s_i$. Suppose
$f_i=[a_{0i}:\cdots:a_{ni}]$. We can
assume $a_{ii}=1$ since $f_i\notin P_i$. We record these
$f_i$'s by the matrix $\A=(a_{ij})$ with $1$ on diagonals and denote the resulting reflection group by $\Gamma_\A$.

Let $H^+\subset \PGL_{n+1}\mathbb{R}$ be the subgroup consisting of positive diagonal matrices $\lambda=\diag(\lambda_0,\cdots, \lambda_n)$, $\lambda_i>0$. So $H^+$ is the identity component of the stabilizer of $P$. Given $\lambda\in H^+$ and a reflection group $\Gamma_\A$ as above, the conjugate $\lambda\Gamma_\A\lambda^{-1}$ is just $\Gamma_{\lambda\A\lambda^{-1}}$.  Put 
$$
\mathcal{M}_{n+1}=\{\A=(a_{ij})\in\mathbb{R}^{(n+1)\times(n+1)}\mid a_{ii}=1\mbox{ for any }i\}.
$$
The quotient $\mathcal{M}_{n+1}/H^+$ by conjugation action  is the moduli space of simplicial reflection groups. We now proceed to discuss discreteness of such groups.

By a \emph{Coxeter diagram with $n$ nodes} we mean a collection of integers $J=(m_{ij})$, where $i,j\in\{0,\cdots, n\}$ are distinct, such that $2\leq m_{ij}\leq \infty$. Note that $m_{ij}=\infty$ is allowed. $J$ is a ``diagram" because we view it as a graph with

\begin{itemize}
\item
$n$ nodes labelled by $0,1,\cdots, n$,
\item 
at most one weigted edge joining any two nodes $i$ and $j$: no edge if $m_{ij}=2$ and an edge of  weight $m_{ij}$ if $m_{ij}\geq 3$.
\end{itemize}

The Coxeter diagram $J=(m_{ij})$ determines an
abstract Coxeter group $W_J$ given by the presentation
$$W_J=\langle\,\tau_0,\cdots,\tau_n\mid(\tau_i\tau_j)^{m_{ij}}=\tau_i^2=1,\, \forall\, i\neq j
\,\rangle.$$
Note that $(\tau_i\tau_j)^\infty=1$ means $\tau_i\tau_j$ has infinite order.

The \emph{Cartan matrix} of $J$, denoted by $\mathbf{C}_J$, is
defined to be the symmetric matrix whose diagonal entries are $1$
and the $(i,j)$-entry is $-\cos(\pi/m_{ij})$ if $i\neq j$.

We have the following sufficient condition on $\A\in\mathcal{M}_{n+1}$ in order that $\Gamma_\A$ is discrete. This is a special case of Theorem 1.5 in \cite{benoist_five}.
\begin{theorem*}[Tits, Vinberg]
Let $\A=(a_{ij})\in\mathcal{M}_{n+1}$. Let $P^\circ$ be the interior of the fundamental simplex $P$. The translates $\gamma(P^\circ)$, $\gamma\in\Gamma_\A$ are pairwise disjoint if and only if there is a Coxeter diagram
$J=(m_{ij})$ such that $\A$ satisfies
the following condition, which we call Condition (J):
\begin{align*}\tag{J}\label{condj}
\begin{cases}
a_{ij}=a_{ji}=0 &\mbox{if } \,m_{ij}=2\\
a_{ij}<0 \mbox{ and } a_{ij}a_{ji}=\cos ^2(\pi/m_{ij})&\mbox{if }\, 3\leq m_{ij}<
\infty\\
a_{ij}<0 \mbox{ and }a_{ij}a_{ji}\geq 1 &\mbox{if }\, m_{ij}=\infty
\end{cases}
\end{align*}

Furthermore, when Condition (\ref{condj}) is satisfied, the following assertions hold:

(1) $\rho_\A: \tau_i\mapsto s_i$ $(0\leq i\leq n)$ is an isomorphism
from $W_J$ to $\Gamma_\A$. Here $s_i\in\PGL_{n+1}\mathbb{R}$ is the reflection fixing the face $P_i$ of $P$ and the point $f_i$ whose coordinates are given by the $i^\mathrm{th}$ column of $\A$.

(2) The set $\Omega_\A=\cup_{\gamma\in\Gamma_\A}\gamma(P)$, called the
\emph{Tits set}, is either the whole $\mathbb{P}^n$ (this occurs if and only if $W_J$ is finite) or a convex
subset in some affine chart of $\mathbb{P}^n$. $\Gamma_\A$ acts properly
discontinuously on $\Omega$.

(3) $\Omega_\A$ is open if and only if the stabilizer in $\Gamma_\A$ of each vertex of
$P$ is finite.
\end{theorem*}
We let $\mathcal{M}_J\subset\mathcal{M}_{n+1}$ denote the set of matrices
satisfying Condition $(\ref{condj})$, which is preserved by the $H^+$-action.

There are only a few choices of $J$ such that $\Omega_\A$ (where $\A\in\mathcal{M}_J$) is open and is not the whole $\mathbb{P}^n$, hence $\Omega_\A/\Gamma_\A$ is a convex projective orbifold. Indeed, $\Omega_\A\neq\mathbb{P}^n$ and $\Omega_\A$ being open are respectively equivalent to following two constraints on the Cartan matrix $\mathbf{C}_J$:
\begin{itemize}
\item $\mathbf{C}_J$ is not positively definite.
\item Every proper principle submatrix of $\mathbf{C}_J$ is
positively definite.
\end{itemize}
Such $\mathbf{C}_J$'s are completely classified. The corresponding $J$'s are divided into 
classes (\textit{c.f.} \cite{vinberg}):

\textbf{Euclidean Coxeter diagrams (\textit{i.e.} $\mathbf{C}_J$ is degenerate):} In this case $\mathbf{C}_J$ has corank
$1$ and there is a faithful representation $\rho_0:W_J\rightarrow
\Isom(\mathbb{E}^n)$ which realize $W_J$ as an Euclidean simplicial
reflection group. All Euclidean Coxeter diagrams are enumerated by Coxeter himself. We conjecture that in this case the Tits set $\Omega_\A$  ($\A\in\mathcal{M}_J$) is either an open simplex containing $P$ or an affine chart.

\textbf{Hyperbolic Coxeter diagrams (\textit{i.e.} $\mathbf{C}_J$ is non-degenerate):} In this
case $\mathbf{C}_J$ has signature $(1,n)$ and there is a faithful
representation $\rho_0:W_J\rightarrow \Isom(\mathbb{H}^n)$ which
realize $W_J$ as a hyperbolic simplicial reflection group. F. Lann\'er  \cite{flanner}
enumerated all hyperbolic Coxeter diagrams as in Figure \ref{lanner} above
(\textit{c.f.} \cite{vinberg}). Note that they exist only for dimension $n\leq 4$. Since $W_J$ is a word-hyperbolic group,  a theorem of Benoist \cite{benoist_1} says that $\Omega_\A$  ($\A\in\mathcal{M}_J$) is
strictly convex. These $\Omega_\A$'s are our main concern in the following sections. Historically,  they provide the first ``non-trivial" example of convex projective structures \cite{kac-vinberg}. See Figure \ref{limit} for some $2$-dimensional examples.

\section{Deformation of simplicial Tits sets}\label{section_moduli}
We fix a hyperbolic Coxeter diagram $J$ and consider $W_J\subset\Isom(\mathbb{H}^n)$ as a hyperbolic reflection group with fundamental simplex $P\subset\mathbb{H}^n$. Put $X_J=\mathbb{H}^n/W_J$. The goal of this section is to prove Proposition \ref{goldman} and Proposition \ref{shape}.

The hyperbolic $n$-space $\mathbb{H}^n$ is a ball in $\mathbb{P}^n$ (the Klein-Beltrami model). Let $P_0,\cdots, P_n$ be the faces of $P$ and $L_i$ be the
hyperplane in $\mathbb{P}^n$ containing $P_i$. Consider a faithful
representation $\rho:W_J\rightarrow\PGL_{n+1}\mathbb{R}$ which
defines a convex projective structure, \textit{i.e.} there is some convex open set
$\Omega_\rho$ and a $\rho$-equivariant homeomorphism $\Phi:\mathbb{H}^n\rightarrow
\Omega_\rho$. Since $\rho(\tau_i)$ has order $2$, its fixed point set in
$\mathbb{P}^n$ is the disjoint union of a $k$-dimensional subspace
and a $(n-k)$-dimensional subspace. On the other hand, $\rho(\tau_i)$ fixes
pointwisely $\Phi(L_i)$, a $(n-1)$-dimensional submanifold of
$\Omega_\rho$, so we conclude that $\Phi(L_i)$ is a hyperplan and $\rho(\tau_i)$ is a
reflection. $\rho(W_J)$ is thus a simplicial projective reflection group. But the discussions in the previous section implies that such groups, up to conjugacy, correspond to matrices $
\A\in\mathcal{M}_J$ up to conjugation by $H^+$. Therefore we get an identification
$$
\mathfrak{P}(X_J)=\mathcal{M}_J/H^+.
$$

%

We are to determined the latter quotient. To this end, for any $(n+1)\times(n+1)$ matrix $\A=(a_{ij})$ and a cyclicly ordered set of
indices $I=(i_1, \cdots, i_k)$, $i_1,\cdots, i_k\in\{0,\cdots, n\}$, we put
$$
\cyp{\A}{I}=a_{i_1i_2}a_{i_2i_3}\cdots a_{i_{k-1}i_k}a_{i_ki_1}.
$$
In particular, $\cyp{\A}{i}=a_{ii}$, $\cyp{\A}{i,j}=a_{ij}a_{ji}$.

\begin{lemma}\label{lemma1} 
Let $\mathcal{M}_{n+1}$ be the set of $(n+1)\times(n+1)$ real matrices with $1$ on diagonals and put
$$
\mathcal{M}_{n+1}^\circ=\{\A=(a_{ij})\in\mathcal{M}_{n+1}\mid \mbox{ for any } i\neq j,\, a_{ij}=0\mbox{ if and only if }a_{ji}=0\}.
$$
Then for any $\A, \mathbf{B}\in\mathcal{M}_{n+1}^\circ$, $\A=\lambda\mathbf{B}\lambda^{-1}$ for some positive diagonal matrix $\lambda=\diag(\lambda_0,\cdots,\lambda_n)$, $\lambda_i>0$ if and only if 
$
\A_I=\mathbf{B}_I
$
for any $I=(i_1, \cdots, i_k)$ with $|I|=k\geq 2$. 
%
%
\end{lemma}
\begin{proof} 
The ``only if" part is elementary and we only treat the ``if" part.
We say that $\A$ is \emph{irreducible} if
it cannot be brought into block-diagonal form by a permutation of basis. 
$\cyp{\A}{ij}=\cyp{\mathbf{B}}{ij}$ implies that $a_{ij}=0$ if only if $b_{ij}=0$. Therefore, after a permutation of basis if necessary,
we can assume that $\A$ and $\mathbf{B}$ are both
block-diagonal with irreducible blocks and that the $r^\mathrm{th}$ block of
$\A$ has the same size with the $r^\mathrm{th}$ block of
$\mathbf{B}$. $\A$ and $\mathbf{B}$ are conjugate
through a diagonal matrix if and only if their blocks are, so we can
assume that $\A$ and $\mathbf{B}$ are irreducible.

We look for $\lambda_1,\cdots,\lambda_n>0$ such that
$\lambda_ia_{ij}\lambda_j^{-1}=b_{ij}$, or equivalently,
\begin{equation}\label{conjugation}
\frac{\lambda_i}{\lambda_j}=\frac{b_{ij}}{a_{ij}}\mbox{  for all  }
i\neq j \mbox{  such that  } a_{ij}\neq 0
\end{equation}

Put $\lambda_1=1$. Irreducibility implies that for each
$i\in \{1, 2, \cdots, n\}$ there is sequence of distinct indices
$1, i_1, i_2, \cdots, i_k, i$, such that $a_{1i_1}$, $a_{i_1i_2}$,$
\cdots$, $a_{i_{k-1}i_k}$, $a_{i_ki}$ are all non-zero. Thus we 
set
\begin{equation}\label{lemma1a}
\lambda_i=\frac{\lambda_i}{\lambda_{i_k}}\frac{\lambda_{i_k}}{\lambda_{i_{k-1}}}\cdots
\frac{\lambda_{i_1}}{\lambda_1}=\frac{b_{ii_k}}{a_{ii_k}}\frac{b_{i_ki_{k-1}}}{a_{i_ki_{k-1}}}\cdots
\frac{b_{i_11}}{a_{i_11}}.
\end{equation}

Let us check that $\lambda_i$ does not depend on the choice of the sequence of indices, namely,  we have
\begin{equation}\label{lemma1b}
\frac{b_{ii_k}}{a_{ii_k}}\frac{b_{i_ki_{k-1}}}{a_{i_ki_{k-1}}}\cdots
\frac{b_{i_11}}{a_{i_11}}=
\frac{b_{ij_m}}{a_{ij_m}}\frac{b_{j_mj_{m-1}}}{a_{j_mj_{m-1}}}\cdots
\frac{b_{j_11}}{a_{j_11}}
\end{equation}
for another sequence $1, j_1, j_2, \cdots, j_m,
i$. Using the hypothesis
$$
\cyp{\A}{i,j}=a_{ij}a_{ji}=b_{ij}b_{ji}=\cyp{\mathbf{B}}{i,j},
$$
we can write the right-hand side of Eq.(\ref{lemma1b}) as
$$
\frac{a_{j_mi}}{b_{j_mi}}\frac{a_{j_{m-1}j_m}}{b_{j_{m-1}j_m}}\cdots
\frac{a_{1j_1}}{b_{1j_1}}.
$$
This coincides the left-hand side because of the equality $\cyp{\A}{I}=\cyp{\mathbf{B}}{I}$ for $I=(1,j_1,\cdots, j_m, i, i_k, i_{k-1},\cdots, i_1)$. A similar equality show that the $\lambda_i$ defined by Eq.(\ref{lemma1a})
$\lambda_i$'s satisfy (\ref{conjugation}).
\end{proof}

\begin{proof}[Proof of Proposition \ref{goldman}] $J=(m_{ij})$ is not circular if and only if  
 $\cyp{\A}{I}=0$ for any $\A\in\mathcal{M}_J$ and $|I|\geq 3$. But we also have $\cyp{\A}{i,j}=a_{ij}a_{ji}=\cos^2(\pi/m_{ij})$. Hence, for a given $I$ with $|I|\geq 2$, $\cyp{\A}{I}$ is the same for any $\A\in\mathcal{M}_J$. By Lemma \ref{lemma1}, any $\A\in\mathcal{M}_J$ are conjugate to each other through $H^+$. 

If $J$ is circular, then any $\A\in\mathcal{M}_J$ looks like the following one (where $n=4$):
$$\A=\left(
\begin{array}{ccccc}
1&a_{01}&0&0&a_{04}\\
a_{10}&1&a_{12}&0&0\\
0&a_{21}&1&a_{23}&0\\
0&0&a_{32}&1&a_{34}\\
a_{40}&0&0&a_{43}&1
\end{array}
\right).$$

Again, given $i$ and $j$, $\cyp{\A}{i,j}$ is the same for any $\A$. The only two non-zero $\cyp{\A}{I}$'s for $|I|\geq 3$ are
$$\cyp{\A}{0,1,\cdots, n}=a_{01}\cdots
a_{n-1,n}a_{n0}\,,\quad\cyp{\A}{n,n-1,\cdots,0}=a_{n,n-1}\cdots a_{10}a_{0n}.$$ 
They determine each other because the product is a constant
$$
\cyp{\A}{0,1,\cdots, n}\cdot\cyp{\A}{n,n-1,\cdots, 0}=\cos^2(\frac{\pi}{m_{01}})\cos^2(\frac{\pi}{m_{12}})\cdots
\cos^2(\frac{\pi}{m_{n0}}). 
$$

Therefore, Lemma \ref{lemma1} implies that $\A\in\mathcal{M}_J$ is determined up to $H^+$-conjugacy by $\cyp{\A}{0,1,\cdots,n}$, which is always positive (resp. negative) if $n$ is odd (resp.
even). Thus we get a homeomorphism
$$
\begin{array}{rcl}
\mathfrak{P}(X_J)=\mathcal{M}_J/H^+&\rightarrow&\mathbb{R}_+\\
\left[\A\right] &\mapsto& |\cyp{\A}{0,1,\cdots,n}|\\
\end{array}
$$
\end{proof}

In order to study how the Tits set deforms when $[\A]$ goes
to $0$ or $+\infty$ in $\mathfrak{P}(X_J)$, we need the follow
lemma, which bounds the Tits set by a simplex. 
\begin{lemma}\label{bound}
Let $J$ be a circular hyperbolic Coxeter diagram and take $\A\in\mathcal{M}_J$.
Let $f_i\in \mathbb{P}^n$ be the point with coordinates
 given by the $i^\mathrm{th}$ column of $\A$.
 Then there is a simplices with vertices
$f_0,\cdots,f_n$ which contains the Tits set $\Omega_\mathbf{A}$.
\end{lemma}
\begin{proof}
Let $L_i$ be the hyperplane of $\mathbb{P}^n$ spanned by
$f_0,\cdots,f_{i-1},f_{i+1},\cdots,f_n$ and $H_i=\{[x_0,\cdots, x_n]\mid x_i=0\}$ be the hyperplane containing the face $P_i$. Assume by contradiction
that $\Omega=\Omega_\A$ is not contained in any simplex with vertices
$f_0,\cdots,f_n$, or equivalently, $\Omega$ meets some $L_i$, say,  $L_0$. 

Recall that $\Omega$ is preserve by the group $\Gamma_\A$, which is in turn generated by $s_0,\cdots, s_n$, where $s_i$ is the reflection with fixed points $\Fix(s_i)=f_i\sqcup H_i$. In general, a reflection $s$ with $\Fix(s)=f\sqcup H$ stabilizes any projective subspace containing $f$. It follows that $L_0$ is stabilized by the subgroup $\Gamma_0\subset\Gamma_\A$ generated by $s_1,\cdots, s_n$. $\Gamma_0$ is a finite group because $J$ is hyperbolic. 

We claim that the vertex $p_0=[1:0,\cdots,0]$ of $P$ is not in $L_0$. Indeed, on one hand, since $J$ is circular, each column of $\A\in\mathcal{M}_J$ has at least two non-zero off-diagonal entries. In other words, each $f_i$ lies outside at least two $H_j$'s, thus
$$
\Fix(\Gamma_0)=(f_1\cup H_1)\cap\cdots\cap(f_n\cup H_n)=H_1\cap\cdots\cap H_n=\{p_0\}.
$$
\textit{i.e.} $p_0$ is the only fixed point of $\Gamma_0$. On the other hand, $\Gamma_0$ preserve the affine chart $\mathcal{A}_0=\mathbb{P}^n\setminus L_0$ and hence has fixed points in $\mathcal{A}_0$, namely, barycenters of orbits. Thus $p_0\in\mathcal{A}_0$ and the claim is proved.

To finish the contradiction argument, we put 
$$
C=\bigcup_{x\in \Omega\cap L_0}[p_0, x],
$$
where $[p_0,x]$ is the segment joining $p_0$ and $x$ within $\Omega$. Namely, $C$ is the cone over $p_0$ generated by
$\Omega\cap L_0$. We consider $\mathcal{A}_0$ as a vector space with origin $p_0$, so that $C$ is a properly convex cone in $\mathcal{A}_0$ (\textit{i.e.} the projectivization of $C$ is properly convex in $\mathbb{P}(\mathcal{A}_0)$).
%
Since $\Gamma_0$ preserves $C$, taking the barycenter of a non-zero $\Gamma_0$-orbit in $C$ gives a fixed point of $\Gamma_0$ different from $p_0$, contradicting the fact $\Fix(\Gamma_0)=\{p_0\}$ which we have established above.
\end{proof}
\begin{proof}[Proof of Proposition \ref{shape}]
We only consider the $n=3$ case to simplify notations. Thus we fix
a circular hyperbolic Coxeter diagram $J=(m_{ij})$ with nodes $\{0,1,2,3\}$. Any
$\A\in\mathcal{M}_J$ has the form
$$\A=
\left(
\begin{array}{cccc}
1&a_{01}&0&a_{03}\\
a_{10}&1&a_{12}&0\\0&a_{21}&1&a_{23}\\
a_{30}&0&a_{32}&1
\end{array}
\right)$$
with $a_{ij}<0$ and $a_{ij}a_{ji}=\cos^2(\pi/m_{ij})$.

We define a one-parameter family of matrices
$\{\A_t\}_{t\in\mathbb{R}}\subset\mathcal{M}_J$ by
$$
\A_t=\left(
\begin{array}{cccc}
1&-t\cos^2(\frac{\pi}{m_{01}})&0&-t^{-1}\\
-t^{-1}&1&-t\cos^2(\frac{\pi}{m_{12}})&0\\
0&-t^{-1}&1&-t\cos^2(\frac{\pi}{m_{23}})\\
-t\cos^2(\frac{\pi}{m_{30}})&0&-t^{-1}&1\\
\end{array}
\right).
$$

Since $|\cyp{\A_t}{0,1,2,3}|=t^4$, by the proof of Proposition
\ref{goldman}, every $\A\in\mathcal{M}_J$ is
$H^+$-conjugate to a unique $\A_t$. We assert that the family of representations $$\rho_t=\rho_{\A_t}: W_J\rightarrow\PGL_{n+1}\mathbb{R}\,, \quad t\in \mathbb{R}_+$$ given by the Tits-Vinberg theorem is the required one.

To see this, let $f_i(t)$ be the point in $\mathbb{P}^n$ with coordinates given by the $i^\mathrm{th}$ column of
$\A_t$ and let $p_0=[1:0:0:0],\cdots,p_3=[0:0:0:1]$ be the vertices of $P$. Each
$f_i(t)$ converges to $p_{i+1}$ when $t\rightarrow0$, and to
$p_{i-1}$ when $t\rightarrow+\infty$ (here the indices are counted
mod $4$). Therefore the simplex bounding $\Omega_t=\Omega_{\A_t}$ given by Lemma
\ref{bound} converges to $P$ in the Hausdorff topology, hence so does $\Omega_t$.
\end{proof}

\section{Metric geometry of simplicial Tits sets}\label{section_metric}
The \emph{Hilbert metric} $d_\Omega$ on a properly convex open set $\Omega\subset\mathbb{P}^n$ is defined as follows. Take any
affine chart $\mathbb{R}^n$ containing the closure $\overline{\Omega}$. For $x,y\in\Omega$,
let $x',y
$ be the points on the boundary $\partial\Omega$ such
that $x',x,y,y'$ lie consecutively on the segment $[x',y']$. We then put
\begin{equation}\label{hilbert metric}
d_\Omega(x,y)=\frac{1}{2}\log
|[x', x, y, y']|,
\end{equation}
where $[x', x, y, y']= \frac{(x'-y)(y'-x)}{(x'-x)(y'-y)}$ is the cross-ratio.

We refer to \cite{harpe} for basic properties of $d_\Omega$. A crucial property which we will use implicitly several times below is that geodesics in $(\Omega,d_\Omega)$ are straight lines.

We shall study the geometry of the
Hilbert metric $d_t=d_{\Omega_t}$ on the Tits set $\Omega_t$ given by Proposition \ref{shape} and constructed in the previous section. The goal of this section is to prove Theorem \ref{thm1} and
Corollary \ref{surface} admitting the technical Lemma \ref{main}.

We begin with the observation that $\Omega_t$ is a simplicial complex whose $k$-cells are translates of
the $k$-cells of $P$ by the $W_J$-action. We denote the $k$-skeleton
of $\Omega_t$ by $\Omega_t^{(k)}$ and let $d_t^{(k)}$ be the
intrinsic geodesic metric on $\Omega_t^{(k)}$ induced by $d_t$, \textit{i.e.}
$$
d_t^{(k)}(x,y)=\min\{l_t(\gamma)\mid \gamma\subset\Omega^{(k)}_t \mbox{ is a piecewise geodesic joining $x$ and $y$} \}.
$$
Here $l_t(\gamma)$ is the length of $\gamma$ measure under $d_t$. 
In particular,
$(\Omega_t^{(1)},d_t^{(1)})$ is a metric graph, whereas $d^{(n)}_t$ is just $d_t$ itself. 

Lemma \ref{main} implies that these metrics are uniformly equivalent to each other:
\begin{lemma}\label{metric comparison}
Suppose $2\leq k\leq n$. There is a constant $C$ depending only on
$J$ such that for any $t\in\mathbb{R}_+$ and $x,y\in
\Omega_t^{(k-1)}$, we have
\begin{equation}\label{equation_compa1}
d_t^{(k)}(x,y)\leq d_t^{(k-1)}(x,y)\leq Cd_t^{(k)}(x,y)
\end{equation}
As a result, iterating the above inequality for $k=2,\cdots, n$, we get 
$$d_t(x,y)\leq d_t^{(1)}(x,y)\leq C'd_t(x,y)$$ for
a constant $C'$ depending only on $J$.
\end{lemma}
\begin{proof}
The first ``$\leq$" in (\ref{equation_compa1}) follows immediately from the definition of the $d^{(k)}_t$'s.

We prove the second ``$\leq$" in (\ref{equation_compa1}). Let
$c:[0,1]\rightarrow\Omega_t^{(k)}$ be a piecewise geodesic joining
$x, y\in\Omega_t^{(k-1)}$ such that the length of $c$ equals
$d_t^{(k)}(x,y)$.

Let $t_0, t_1, t_2, \cdots, t_r\in[0,1]$ with $t_0=0$ and $t_r=1$ be such that each
$c([t_{i-1},t_i])$ lies in a single $k$-cell and that the $c(t_{i})$'s
are in $\Omega_t^{(k-1)}$. Since $c$ is length-minimizing, each
$c[t_{i-1},t_i]$ must be a segment, whose length equals the distance
between then two end points. Thus if we could prove
$$d_t^{(k-1)}(c(t_{i-1}),c(t_i))\leq Cd_t^{(k)}(c(t_{i-1}),c(t_i))
$$ then we can take the sum
over $1\leq i\leq r$ and use the triangle inequality to obtain
$$d_t^{(k-1)}(x,y)\leq Cd_t^{(k)}(x,y).$$

Therefore, we can assume that both $x$ and $y$ lie on the boundary
of a $k$-cell. Since each $k$-cell is isometric to some sub-cell of
$P$, it is sufficient to prove that, for any $k$-dimensional sub-cell
$F$ of $P$ we have 
$$
d_t^{(k-1)}(x,y)\leq Cd_t^{(k)}(x,y)=Cd_t(x,y).
$$
for $t\in\mathbb{R}_+$ and $x,y\in F$.

If $x, y$ both lie on the same $(k-1)$-dimensional sub-cell of $F$,
then we have $d_t^{(k)}(x,y)= d_t^{(k-1)}(x,y)$ and there is nothing
to prove. So we assume that $x$ and $y$ belong to $(k-1)$-dimensional sub-cells $A$ and $B$, respectively. $E=A\cap B$
is a $(k-2)$-dimensional sub-cell. Let $x_0, y_0$ be a point in $E$ nearest to $x$, $y$, respectively, \textit{i.e.} $d_t(x, E)=d_t(x,
x_0)$ and $d_t(y, E)=d_t(y, y_0)$.

The three segments $[x,x_0]$, $[x_0,y_0]$ and $[y_0,y]$ lie in
$\Omega_t^{(k-1)}$ and form a piecewise segment joining $x, y$, so the definition of $d_t^{(k-1)}$ implies \begin{equation}\label{inequ1}
d_t^{(k-1)}(x,y)\leq d_t(x,x_0)+d_t(x_0,y_0)+d_t(y_0,y).
\end{equation}
By the triangle inequality, we have
\begin{equation}\label{inequ2}
d_t(x_0,y_0)\leq d_t(x_0,x)+d_t(x,y)+d_t(y,y_0).
\end{equation}
(\ref{inequ1}) and (\ref{inequ2}) gives
$$
d_t^{(k-1)}(x,y)\leq 2(d_t(x,x_0)+d_t(y,y_0))+d_t(x,y)=2(d_t(x,E)+d_t(y,E))+d_t(x,y)
$$
Now we apply Lemma \ref{main}, and conclude that
$$
d_t^{(k-1)}(x,y)\leq (2C+1)d_t(x,y)
$$
this is the required inequality.
\end{proof}

\begin{proof}[Proof of Theorem 1]Note that each vertex of the simplex $P$ lies on different
orbits of $W_J$, so the vertex set $\Omega_t^{(0)}$ is the union of
$n+1$ orbits. Hence, fixing any vertex $v_0$, we have the follow
expression for the entropy $\delta_t=\delta(\Omega_t,d_t,W_J)$:
$$
\delta_t=\varlimsup_{R\rightarrow\infty}\frac{1}{R}\log\#\{v\in\Omega_t^{(0)}|d_t(v,v_0)\leq
R\}.
$$

We shall compare $\delta_t$ with the entropy $\delta_t^{(1)}=\delta(\Omega_t^{(1)},d_t^{(1)},W_J)$ of the metric graph
$(\Omega_t^{(1)},d_t^{(1)})$, which is defined by 
$$\delta_t^{(1)}=\varlimsup_{R\rightarrow\infty}\frac{1}{R}\log\#\{v\in\Omega_t^{(0)}|d_t^{(1)}(v,v_0)\leq
R\}
$$
The comparison of $d_t^{(1)}$ and $d_t$ given by Lemma \ref{metric comparison} implies there is a
constant $C'$ depending only on $J$ such that
$$
\delta_t^{(1)}\leq \delta_t\leq C'
\delta_t^{(1)}.
$$
So it is sufficient to prove
$$\delta_t^{(1)}\rightarrow 0 \quad\mbox{ as  } t\rightarrow0\mbox{ or }t\rightarrow+\infty.$$

To this end, let $m(t)$ be the minimal length of edges of $P$ under
$d_t$. We have seen in Proposition \ref{shape} that $\Omega_t$
approaches the simplex $P$ when $t\rightarrow0$ or $+\infty$. Using
the expression of Hilbert metric (\ref{hilbert metric}) one can see
the length of each edge of $P$ tends to $+\infty$, thus
$m(t)\rightarrow +\infty$.

On the other hand, a $W_J$-invariant geodesic metric on the graph
$\Omega_t^{(1)}$ is uniquely determined by lengths of the edges of
$P$ and is monotone with respect to each of these lengths. Therefore,
if we let $d'$ be the metric on the graph defined by setting all edge lengths to
be $1$, then we have $d_t^{(1)}\geq m(t)d'$. This allows us to compare the entropy $\delta_1^{(1)}$ defined by $d_t^{(1)}$ with the one defined by $d'$:
$$
\delta_t^{(1)}\leq \frac{1}{m(t)}\delta(\Omega_t^{(1)},d',W_J).
$$
But the right-hand side tends to $0$ because $\delta(\Omega_t^{(1)},d',W_J)$ is a constant.
\end{proof}

\begin{proof}[Proof of Corollary \ref{surface}]
Let $\Sigma$ be a surface with genus $\geq 2$. We claim that there are integers $p,q,r\geq 3$ with
$\frac{1}{p}+\frac{1}{q}+\frac{1}{r}<1$ and a subgroup  $\Pi$ of finite index in the
$(p,q,r)$-triangle group $\Delta=\Delta_{p,q,r}$ such that $\Pi$
acts freely on the hyperbolic plan $\mathbb{H}^2$ with quotient 
$\mathbb{H}^2/\Pi\cong\Sigma$.  Restricting the one-parameter family of representations
$\rho_t:\Delta\rightarrow\PGL_{n+1}\mathbb{R}$ given by Proposition \ref{goldman} and
\ref{shape} to $\Pi$, we obtain an one-parameter
family of convex projective structures on $\Sigma$. We shall show that this family fulfils the requirements. 

Since $\Pi\subset\Delta$ has finite index, the entropy  $\delta(\Omega_t,d_t,\Pi)$ of the convex projective surface equals the entropy $\delta(\Omega_t,d_t,\Delta)$ which tends to $0$ by Theorem \ref{thm1}.

Lemma \ref{main} implies that for any $t$, every triangle
inscribed in the fundamental triangle $P$ has perimeter greater than
$\frac{1}{C}$ times the perimeter of $P$ (measure by $d_t$). But the latter perimeter tends to $+\infty$ because of convergence of $\Omega_t$ to $P$. Thus the constant of Gromov hyperbolicity of
$\Omega_t$ tends to $+\infty$.

To show that the systole tends to $+\infty$, we take a
homotopically non-trivial closed curve $c$ which is the shortest
under $d_t$. The image of $c$ under the orbifold covering map
$\Sigma\cong\mathbb{H}^2/\Pi\rightarrow \mathbb{H}^2/\Delta\cong P$ is a
closed billiard trajectory in the triangle $P$ which hits each of the three sides. The same argument as in the previous paragraph shows that the length of $c$ goes to $+\infty$.

Finally, we prove the claim using an explicit constructions. In the picture below,
\begin{figure}[h]
\centering
\includegraphics[width=2.5in]{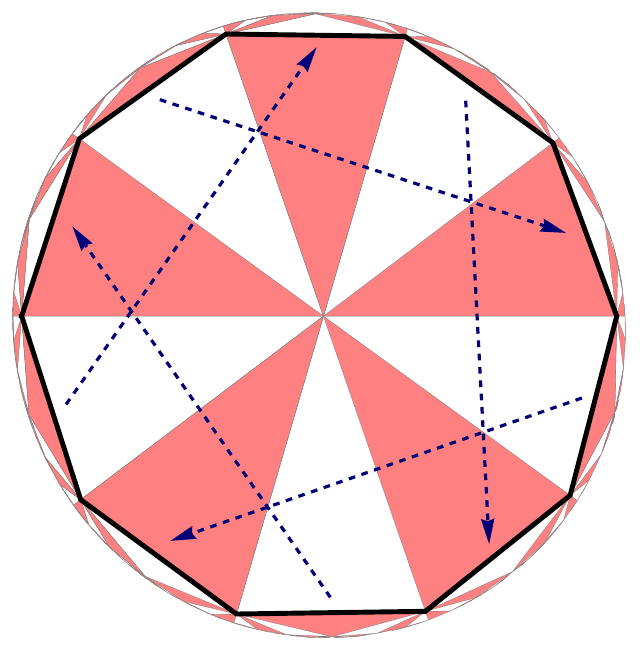}
\end{figure}
The boldfaced $10$-gon consists of ten fundamental domains of the triangle group $\Delta=\Delta_{5,5,5}$. We take the five elements in $\Delta$ indicated by the arrows, each of them pushing the
$10$-gon to an adjacent one. One checks that the group $\Pi$
generated by them has the $10$-gon as a fundamental
domain. The quotient $\mathbb{H}^2/\Pi$ is a surface obtained by
pairwise gluing edges of the $10$-gon. A calculation of Euler
characteristic shows $\mathbb{H}^2/\Pi$ have genus $2$. Since closed
surfaces of higher genus covers the surface of genus $2$, by taking
subgroups of $\Pi$, we conclude that all surfaces of genus $\geq 2$
is the quotient of $\mathbb{H}^2$ by some subgroup of $\Delta$, and the claim is proved.
\end{proof}

\section{Proof of Lemma \ref{main}}\label{section_proof}

To begin with, we need the following fact concerning the cellular structure of $\Omega_t$. 
Looking at Figure \ref{limit},
one observes that the $1$-skeleton of $\Omega_t$ consists of
straight lines. More generally, in higher dimensions, the $k$-skeleton $\Omega_t^{(k)}$ is also a union of $k$-dimensional subspaces\footnote{By a ``subspace" of $\Omega_t$, we mean the intersection of a
projective subspace of $\mathbb{P}^n$ with $\Omega_t$.}, or equivalently, the $k$-dimensional subspace $L$
containing some $k$-cell must be an union of $k$-cells. This can be
proved using the fact that the tangent space of a vertex in
$\Omega_t$ has the structure of a finite Coxeter complex, and it is
well-known that the above statement holds for finite Coxeter complex
(see \textit{e.g.} \cite{humphreys}). We omit the details.

We first present a proof of Lemma \ref{main} for the
simplest $2$-dimensional case, since the main idea is most transparent in this
case.
\begin{proof}[Proof of Lemma \ref{main} for $n=2$]
We may assume
$$W_J=\langle\,\tau_1,\tau_2,\tau_3\mid(\tau_1\tau_2)^p=(\tau_2\tau_3)^q=(\tau_3\tau_1)^r=\tau_1^2=\tau_2^2=\tau_3^2=1\,\rangle.$$

Suppose $x$ and $y$ lie on the sides $A$ and $B$ of a triangle $P$
in $\mathbb{P}^2$, respectively. Denote the common vertex of $A$ and
$B$ by $E$. We need to prove that $$Cd_t(x,y)\geq d_t(x,E)+d_t(y,E),\quad\forall t$$

Put $s_1=\rho_t(\tau_1)$ and
$s_2=\rho_t(\tau_2)$. So $s_1$ and $s_2$ are reflections with respect to $A$ and
$B$, respectively, whereas $s_1s_2$ is a rotation of order $p$.

If $p$ is odd, we put 
$$y'=\underbrace{s_2s_1s_2\cdots s_1}_{p-1\mbox{ reflections }}(y).$$ 
Then $y'$ lies on the opposite half ray of the geodesic ray $\vec{Ex}$ (see Figure
\ref{odd case}). On the other hand, the successive images of $[x,y]$
by the sequence of projective transforms
$$s_2, \ s_2s_1, \ s_2s_1s_2,\ \cdots, \ \underbrace{s_2s_1s_2\cdots
s_1}_{p-1}$$ consititue a piecewise geodesic $\gamma$ joining $x$ and $y'$. $\gamma$ 
consists of $p$ pieces, each one with the same length
$d_t(x,y)$. Thus we have 
$$
p\,d_t(x,y)\geq
d_t(x,y')\geq d_t(x,E).
$$

When $p$ is even, we obtain the above inequality with $x'=s_2s_1\cdots s_2(x)$ replacing $y'$ in the same way (see Figure \ref{even case}).

Interchanging the roles of $x$ and $y$, we get 
$$
p\,d_t(x,y)\geq d_t(y,E)
$$
and conclude that
$$
2p\,d_t(x,y)\geq d_t(x,E)+d_t(y,E).
$$
\begin{figure}[h]
\begin{minipage}[t]{0.45\linewidth}
\centering
\includegraphics[width=2in]{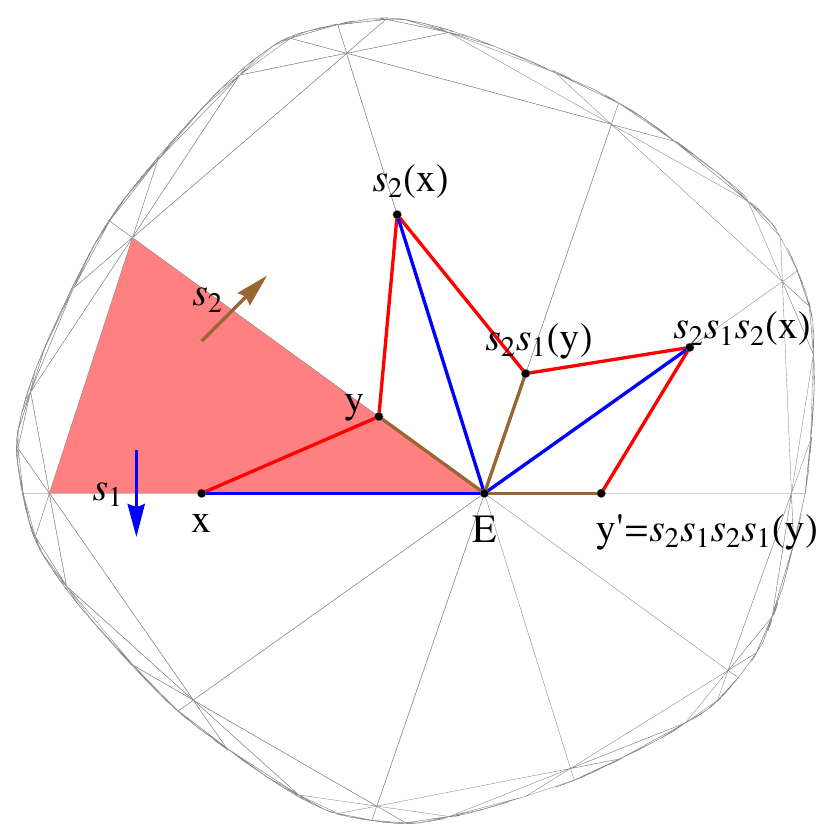}
\caption{$p=5$} \label{odd case}
\end{minipage}
\begin{minipage}[t]{0.45\linewidth}
\centering
\includegraphics[width=2in]{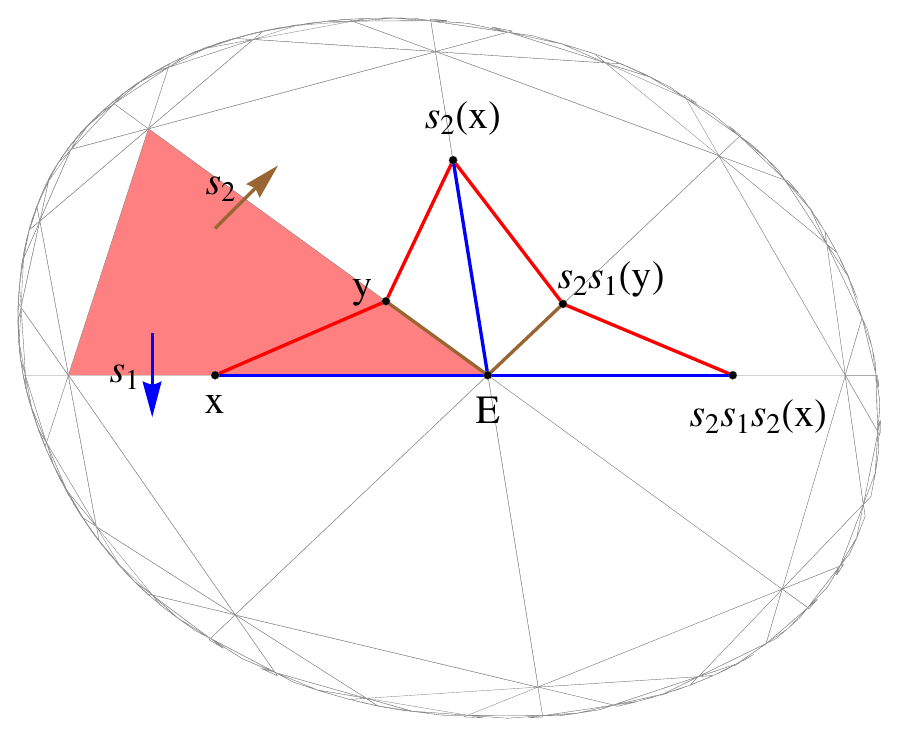}
\caption{$p=4$} \label{even case}
\end{minipage}
\end{figure}
\end{proof}

To tackle the higher-dimensional case, we introduce the following terminology. Let $E$ be a $(k-1)$-cell of
$\Omega_t$. We say two $k$-cells are \emph{$E$-colinear}, if they
lie on the same $k$-dimensional subspace and their intersection is
$E$.  As explained in the beginning of this section, the
$k$-dimensional subspace of $\Omega_t$ containing a $k$-cell $A$ is
a union of $k$-cells, so for any $(k-1)$-sub-cell $E$
of $A$, there is an unique $k$-cell which is $E$-colinear to $A$.

The crucial point of the above proof in dimension $2$ is the following fact: let $V$ be
the $k$-cell $E$-colinear to $A$. Then we can connect $x\in A$ and
some point in $V$ by a curve piecewise isometric to
the segment $[x,y]$, where the number of pieces is bounded by a combinatorial constant. 

In higher dimensions, the situation is more delicate: the cell $V$
which is $E$-colinear to $A$ may not be a translate of $A$ or $B$. This prevents us from constructing a curve going from $x$ to $V$ which is piecewise isometric to
$[x,y]$. Instead of this, we shall take a cell
$A'=\rho_t(\gamma) A$, the translate of $A$ by some
$\gamma\in W_J$, such that $A'$ and $V$ are contained in the same
top-dimensional cell. Now we can go from $x$ to $A'$ along a curve
piecewise isometric to $[x,y]$. To prove Lemma \ref{main}, we then
need to show that the distance from $x$ to $A'$ is greater than the
distance from $x$ to $E$. In order to do this, we will develop some
lemmas concerning distance comparisons in Hilbert geometry. But before going into that, the reader might find it useful to keep in mind the following typical example of the above situation: take $n=3$ and $k=1$ such that $E$ is a vertex while $A$ and $V$ are edges. Assume that the sub-diagram in $J$ corresponding to $E$
 is \footnote{ Given a sub-cell $E$ of the simplex $P$, the sub-diagram of $J$ corresponding to $V$ is the Coxeter diagram consisting of nodes $i$ satisfying $E\subset P_i$ and the same weights as in $J$.}  $\xy (-7,0)*{\circ}="a1";(0,0)*{\circ}="a2";(7,0)*{\circ}="a3";{\ar @{-}"a1";"a2"};{\ar @{-}"a2";"a3"}\endxy$, whose first two nodes form the sub-diagram corresponding to $A$.  The $3$-cells containing $E$ then form the configuration of the barycentric subdivision of a tetrahedron, as partly shown in the picture below. Here $P$ is the tetrahedron $abcE$, while $A$ and $V$ are the segments $[a,E]$ and $[v, E]$, respectively. We can take $A'$ to be either $[a_1,E]$, $[a_2,E]$ or $[a_3,E]$. 
\begin{figure}[h]
\centering
\includegraphics[width=1.6in]{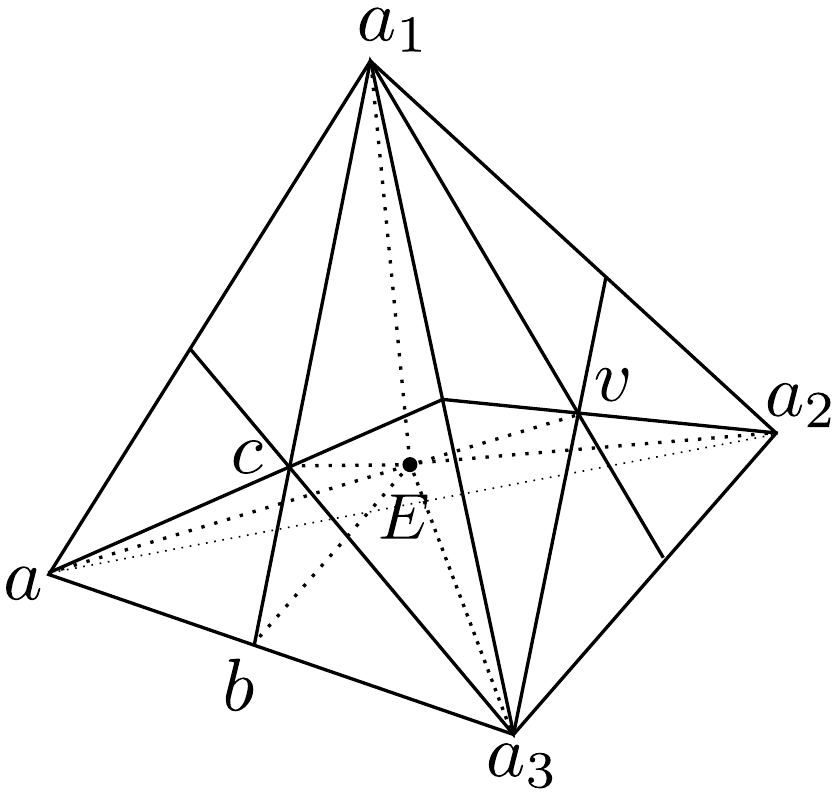}
\end{figure}

Now we discuss distance comparison in Hilbert geometry. Using the definition of Hilbert metric (\ref{hilbert metric}), it
can be shown that if $\Omega\subset\mathbb{P}^n$ is a properly
convex open set which is strictly convex (this is the case for our $\Omega_t$ in question, \textit{c.f.} Section
$2$), then the Hilbert metric $d_\Omega$ has the following property.
Let $L\subset\Omega$ be a subspace and $x\in\Omega\setminus L$. Among all points of $L$, there
is an unique $x_0\in L$ whose distance to $x$ is minimal. We
call $x_0$ the \emph{projection} of $x$ on $L$, and denote it by
$x_0=\Pr_L(x,L)$.

Let $L\subset\Omega$ be a hyperplane, \textit{i.e.} subspace of codimension $1$. We say that $\Omega$ have reflectional symmetry $s$ with
respect to $L$ if $s\in\PGL_{n+1}\mathbb{R}$ is a reflection
preserving $\Omega$ and fixing each point of $L$.  In this case, the
triangle inequality and the fact that geodesics are straight lines
yields the following simple characterization of projection:
\begin{equation}\label{hilbert_projection}
\Pr(x,L)=[x,s(x)]\cap L
\end{equation}
\begin{lemma}\label{projectionlemma}
Let $\Omega\subset\mathbb{P}^n$ be a properly strictly convex open
set with reflectional symmetry $s$ with respect to a hyperplane $L$.
Then for any $x,y\in\Omega$, we have
$$
d_\Omega(\Pr(x,L),\Pr(y,L))\leq d_\Omega(x,y).
$$
In particular, if $x\in L$, then for any $y\in\Omega$ we have
$$
d_\Omega(x,\Pr(y,L))\leq d_\Omega(x,y).
$$
\end{lemma}
\begin{figure}[h]
\centering
\includegraphics[width=3.1in]{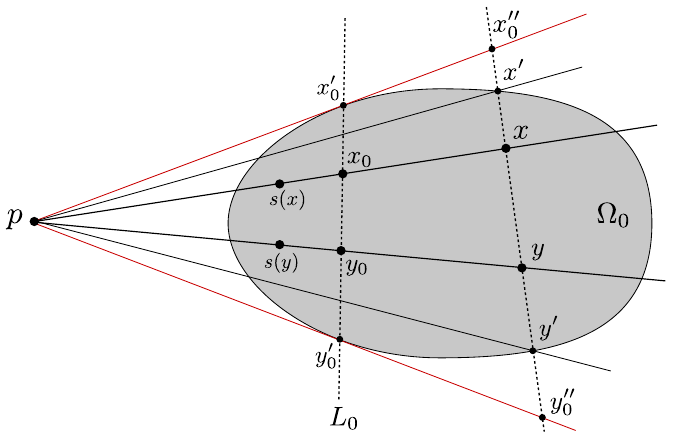}
\caption{$d_\Omega(x_0,y_0)\leq d_\Omega(x,y)$} \label{projection}
\end{figure}
\begin{proof}
Denote $x_0=\Pr(x,L)$ and $y_0=\Pr(y,L)$. The reflection $s$ has
another fix point $p\in\mathbb{P}^n$ outside $L$. We have 
$x_0=\overline{xp}\cap L$ and $y_0=\overline{yp}\cap L$. Therefore the four points $x, y, x_0, y_0$ lie on
the plane $\overline{pxy}$, which is preserved by $s$. So we are reduced to the $2$-dimensional case by restricting the consideration to $\Omega_0=\overline{pxy}\cap\Omega$ and $L_0=\overline{pxy}\cap L$.

Suppose that $L_0$ intersects $\partial\Omega_0$ at $x_0'$ and 
$y_0'$, see Figure
\ref{projection}. Since $\Omega_0$ has reflectional symmetry with respect to
$L_0$, the lines $\overline{px_0'}$ and $\overline{py_0'}$ are tangent
to $\Omega$. Let $x_0''$ (resp. $y_0''$) be the intersection of $\overline{px_0'}$ (resp. $\overline{py_0'}$) with $\overline{xy}$. It is a basis fact from projective geometry that we have an equality of cross-ratios 
$$
[x_0'\,,x_0\,,y_0\,,y_0']=[x_0''\,,x\,,y\,,y_0''].
$$

Since $x'$ and $y'$ lie strictly inside the segment $[x_0'', y_0'']$, we have
$$
\big|[x_0'\,,x_0\,,y_0\,,y_0']\big|=\big|[x_0''\,,x\,,y\,,y_0'']\big|\leq\big|[x'\,,x\,,y\,,y']\big|.
$$
It follows that $d_\Omega(x_0,y_0)\leq d_\Omega(x,y)$.
\end{proof}
 
When there are several reflectional symmetries, we can strengthen the above lemma in the following way. This is the nontrivial ingredient that we need in generalizing the above proof Lemma \ref{main} to higher dimension.
\begin{lemma}\label{projectionlemma2}
Let $\Omega\subset\mathbb{P}^n$ be a properly strictly convex open
set with reflectional symmetries $s_1, \cdots, s_m$ ($m\leq n-1$) with respect to
hyperplanes $L_1, L_2, \cdots, L_m$, such that the $s_i$'s
generate a finite group $\Gamma$. Assume that $W=L_1\cap\cdots\cap L_m$
has dimension $n-m$ (\textit{i.e.} the $L_i$'s are in general position) and $W\cap\Omega\neq\emptyset$. Let $D$ be a
$\Gamma$-invariant convex subset of $\Omega$.

Then for any $x\in W\cap\Omega$ and any $x'\in D$, there
is some point $x_0\in W\cap D$ such that
$$
d_{\Omega}(x,x_0)\leq d_\Omega(x,x')
$$
\end{lemma}

\begin{proof}
Fix $x\in W\cap\Omega$ and $x'\in W\cap D$. We choose an affine chart
$\mathcal{A}\subset\mathbb{P}^n$, an origin $x_0\in\mathcal{A}$ (so as to consider $\mathcal{A}$ as a vector space) and a Euclidean scalar product on $\mathcal{A}$ such that 

(1) $\mathcal{A}$ contains the closure of $\Omega$;

(2) $L_1, \cdots, L_m$ are linear subspaces of $\mathcal{A}$, \textit{i.e.} $x_0\in W$;

(3) $\Gamma$ preserves the Euclidean scalar product;

(4) $x'\in W^\bot$, where $W^\bot$ is the orthogonal complement of
$W$.

Our aim is to show that the origin $x_0$ is contained in $D$ and satisfies the required
inequality. Let us denote $x_0$ simply by $0$. We will mainly work on the subspace $W^\bot$ of $\mathcal{A}$.  Each $L_i'=L_i\cap
W^\bot$ is a subspace of $W^\bot$ of codimension $1$, and the
intersection $L_1'\cap\cdots\cap L_m'=\{0\}$. Since $D\cap W^\bot$
is $\Gamma$-invariant and  convex, the barycenter of the
$\Gamma$-orbit of $x$ lies in $D\cap W^\bot$ and is fixed by
$\Gamma$. But $L_1'\cap\cdots\cap L_m'=\{0\}$ implies that the only fixed
point of $\Gamma$ in $W^\bot$ is $0$, thus $0\in D$.

To prove $d_\Omega(x,0)\leq d_\Omega(x,x')$, we put
$$
C_i=\{y\in W^\bot|\angle(y,L_i')\geq\theta \mbox{ or } y=0\}
$$
where $\angle(y,L_i')$ is the Euclidean angle. Namely, $C_i$ consists of vectors which are apart from $L_i'$ by an angle $\theta$. We take 
$\theta$ small enough so that the union $C_1\cup\cdots\cup C_m$ is the whole $W^\bot$. Any
$y\in C_i$ verifies
$$
|\Pr\nolimits_{W^\bot}(y,L_i')|\leq|y|\cos\theta
$$
where $\Pr_{W^\bot}(y,L_i')$ is the usual Euclidean projection and coincides with the projection $\Pr(y,L_i)$ in
the sense of Hilbert geometry described earlier.

We construct a sequence of points $x'=y_0,\,y_1,\,y_2,\,\cdots\in 
W^\bot\cap D$ converging to $0$ as follows. By recurrence, assume that we already have $y_k$. Since
$C_1\cup\cdots\cup C_m=W^\bot$, there is some $C_{i_k}$ containing
$y_k$. We put $y_{k+1}=\Pr_{W\bot}(y_k, L_{i_k}')$. 

The above inequality yields
$$
|y_k|\leq|y_{k-1}|\cos\theta\leq\cdots \leq|y_0|(\cos\theta)^{k}.
$$
Hence $
\lim_{k\rightarrow\infty}y_k=0$.

We have $y_{k+1}=\Pr(y_k, L_{i_k})$ as mentioned above, so Lemma
\ref{projectionlemma} implies
$$d_{\Omega}(x,y_k)\leq d_\Omega(x,y_{k-1})\leq\cdots\leq
d_\Omega(x,y_0)=d_\Omega(x,x') .$$ Therefore, by the continuity of
$d_\Omega$, we conclude that
$$
d_\Omega(x,0)=\lim_{k\rightarrow\infty}d_\Omega(x,y_k)\leq d_\Omega(x,x').
$$
\end{proof}
Return to the particular convex set $\Omega_t$. For any cell
$V$ of $\Omega_t$, we let  $\St(V)$ denote the union of all closed $n$-cells containing $V$ (``$\St$" for ``star-like" or ``saturated").
\begin{lemma}\label{convexity}
$\St(V)$ is a convex subset of $\Omega_t$.
\end{lemma}
\begin{proof}
Let $F$ be a $(n-1)$-cell on the boundary of $\St(V)$ and
let $L$ be the hyperplane containing $F$. $L$ does not contain $V$,
so $V$ is contained in one of the two ``half-spaces" in 
$\Omega_t$ bounded by $L$. Using the fact that $L$ is an union of
$(n-1)$-cells, we conclude the whole $\St(V)$ lie in the same
half-space as $V$. Therefore, $\St(V)$ is an intersection of half-spaces, hence
convex.
\end{proof}

\begin{proof}[Proof of Lemma \ref{main}]
Fix a hyperbolic Coxeter diagram $J$. The Coxeter group is $$W_J=\langle\,\tau_0,\cdots,\tau_n\mid(\tau_i\tau_j)^{m_{ij}}=\tau_i^2=1,
\forall i\neq j\,\rangle.$$
By definition of hyperbolic Coxeter diagrams, for any $I\subsetneqq\{0,\cdots,n\}$, the
subgroup $W_{I}$ generated by $\{\tau_i\}_{i\in I}$ is finite. The \emph{word-length} of $\sigma\in W_{I}$, denoted by  $l(\sigma)$,  is defined as the minimum of the integer $k$ such that $\sigma=\tau_{i_1}\cdots\tau_{i_k}$ for some $i_1,\cdots, i_k\in I$. We define the \emph{word-length-diameter} of $W_{I}$ as 
$$\diam(W_{I})=\max_{\sigma\in W_{I}}l(\sigma)$$ 
and put 
$$
C=\max_{I\subsetneqq J}\diam(W_I)\,.
$$

We shall show that
$C\,d_t(x,y)\geq d_t(x, E)$. Then exchanging the roles of $x$ and
$y$ we get $Cd_t(x,y)\geq d_t(y, E)$ and the required inequality follows.

Let $V$ be the $k$-cell which is $E$-colinear to $A$. Let $\gamma\in \mbox{Stab}_{W_J}(E)$ be such that
 $P'=\rho_t(\gamma)P$ contains $V$. We
denote $A'=\rho_t(\gamma)A$ and $x'=\rho_t(\gamma)x$.

We first show that there is a curve joining $x$ and $x'$ which is
piecewise isometric to $[x,y]$, with number of pieces at most $C$.

Denote $s_i=\rho_t(\tau_i)$. Let $J_E\subset\{0,1,\cdots,n\}$ be the set of indices
of those $P_i$ such that $E\subset P_i$. Then $J_E$ has $n-k+1$ elements and
$\mbox{Stab}_{W_J}(E)$ is generated by $\{s_i\}_{i\in J_E}$.

We can write $\gamma=\tau_{i_1}\cdots\tau_{i_m}$, with
$i_1,\cdots,i_m\in J_E$ and $m\leq\diam(W_{J_E})\leq C$. Then $\rho_t(\gamma)=s_{i_1}s_{i_2}\cdots s_{i_m}$.
Consider the sequence of
segments
$$
s_{i_1}([x,y]),\ s_{i_1}s_{i_2}([x,y]),\ \cdots, \ s_{i_1}s_{i_2}\cdots
s_{i_m}([x,y]).
$$
The $k$-cell $A$ contains the $(k-1)$-cell $E$, so $A$ has only one
vertex $a$ lying outside $E$. Similarly $B$ has only one
vertex $b$ outside $E$. Each face $P_i$ of $P$ must contain at least
one of the two points $a$ and $b$. Hence a face containing $E$
also contains either $A$ or $B$.  It follows that if $i\in J_E$ then $s_i$
fixes either $x$ or $y$. Therefore, each segment in the above sequence
shares at least one end point with the next one. So the union of
these segments is connected, and we can extract a subset of these
segments to form a curve joining $x$ and $x'$ which
is piecewise isometric to $[x,y]$. The number of pieces is bounded by $C$.

Next, by the triangle inequality, we conclude that
$$
Cd_t(x,y)\geq d_t(x,x').
$$
So it remains to be shown that
\begin{equation}\label{final}
d_t(x,x')\geq d_t(x,E).
\end{equation}

To this end, we use Lemma \ref{projectionlemma2}. Let $\St(V)$ be the convex
set $D$ in Lemma \ref{projectionlemma2},
which contains $x'$ (in the above example $\St(V)$ is the tetrahedron $a_1a_2a_3E$). Let $J_A\subset \{0,1,\cdots,n\}$ be the set of
indices of those faces $P_i$ which contain $A$ and let $L_i$
be the hyperplane containing $P_i$.  $W=\cap_{i\in J_A}L_i$
is the $k$-dimensional subspace containing $A$ and $V$ (in the above example $W$ is the line passing through $a$, $E$ and $v$). For each $i\in J_A$, the reflection $s_i$ preserves
$\St(V)$ since $V\subset L_i$. Thus the hypothesis of Lemma \ref{projectionlemma2} is
verified and we conclude that there is $x_0\in V=\St(V)\cap W$ such that
$$
d_t(x,x')\geq d_t(x,x_0).
$$
The union $A\cup V$ is convex because it is the intersection of $\St(E)$ and a $k$-dimensional
subspace. So $[x,x_0]$ intersects $E$ at some
point $x_1$ and we have
$$
d_t(x,x_0)\geq d_t(x,x_1)\geq d_t(x,E)
$$
Hence we have obtained (\ref{final}), and the proof is complete.
\end{proof}

\bibliographystyle{plain}
\bibliography{convex}

\end{document}